\documentclass[11pt]{amsart}
\usepackage{geometry}
\usepackage{palatino}
\usepackage[all]{xy}
\usepackage{amsthm,amssymb,amsthm,amsmath,graphicx,color}
\usepackage{mathrsfs,amscd}
\usepackage{multirow}
\newtheorem{theorem}{Theorem}[section]
\newtheorem{conjecture}[theorem]{Conjecture}

\newtheorem{corollary}[theorem]{Corollary}
\newtheorem{lemma}[theorem]{Lemma}
\newtheorem{proposition}[theorem]{Proposition}

\newtheorem{remark}[theorem]{Remark}
\newtheorem{definition}[theorem]{Definition}
\newcommand{\bb}[1]{\mathbb{#1}}
\newcommand{\mc}[1]{\mathcal{#1}}
\newcommand{\mf}[1]{\mathfrak{#1}}

\title[A Casselman-Osborne Theorem for rational Cherednik algebras]{A Casselman-Osborne Theorem
for rational Cherednik algebras}
\author{Jing-Song Huang and Kayue Daniel Wong}
\address[Huang]{Department of Mathematics, Hong Kong University of Science and
technology, Clear Water Bay, Kowloon, Hong Kong SAR, China}
\email{mahuang@ust.hk}
\address[Wong]{Department of Mathematics, Hong Kong University of Science and
technology, Clear Water Bay, Kowloon, Hong Kong SAR, China}
\email{makywong@ust.hk}
\thanks{The authors would like to thank Stephen Griffeth and the anonymous referees for valuable comments on an earlier version of this manuscript. The research described in this paper is supported by grants from
Research Grant Council of HKSAR and National Science Foundation of China.}

\begin{document}

\begin{abstract}
We define Lie algebra cohomology associated with the half-Dirac operators for
representations of rational Cherednik algebras
and show that it has property described in the Casselman-Osborne Theorem
by establishing a version of the Vogan's conjecture for the half-Dirac operators.
Moreover, we study the relationship between Lie algebra cohomology
and Dirac cohomology in analogy of the representations for semisimple Lie algebras.
\end{abstract}

\maketitle

\section{Introduction}
Dirac operator plays a pivotal role in mathematics and  theoretical physics. In representation theory, Dirac operator was used
 for geometric construction of discrete series representations by Parthasarathy \cite{P}, Atiyah and Schmid \cite{AS}.  In later 1990's Vogan \cite{V} formulated a conjecture on the Dirac operator in Lie algebra setting that reveals an interesting algebraic nature of the Dirac operator.  This conjecture was verified by Pand\v zi\'c and the first named author \cite{HP2}.
The Vogan's conjecture has been generalized to Kostant's cubic Dirac operator \cite{Ko3}, as well as various other setting in  affine Lie algebras \cite{KMP},  Lie superalgebras \cite{HP3},  and
in particular to graded affine Hecke algebras \cite{BCT}. More recently, Ciubotaru has extended the definition of Dirac operator and Vogan's conjecture further to Drinfeld's graded Hecke algebras including symplectic reflection algebras \cite{C1}.
The Dirac cohomology $H_D(M)$ of an irreducible Harish-Chandra module $M$ determines the infinitesimal character of $M$.
Determination of Dirac cohomology $H_D(M)$ leads insight to many classical subjects such as Lie algebra cohomology
\cite{HP1}, and it has applications in branching rules \cite{HPZ} and harmonic analysis \cite{H}.

The work of Etingof and Ginzburg on symplectic reflection algebras  \cite{EG}
 has inspired interaction of representation theory with
 algebraic combinatorics, integrable systems and algebraic geometry.
Apart from a finite list of small rank exceptions, there are only two families of symplectic reflection algebras:
 one is associated with real or complex refection groups, namely rational Cherednik algebras and the other is the wreath product.
 The representation theory of rational Cherednik algebras has more analogue of semisimple Lie algebras, and we
 focus on this case.
Let $\mf h$ be a finite dimensional $\bb C$-vector space and $W\subset GL(\mf{h})$ a complex reflection group.
The rational Cherednik algebras ${\bf H}_{t,c}(W,\mf h)$ with parameter  $t\in \bb C$ and $W$-invariant functions $c$ defined on the set of reflections of $W$
 is the symplectic reflection algebras with $W$ acting on $V=\mf h\oplus\mf h^*$ with the naturally defined $W$-invariant symplectic form (cf. Section \ref{sec:halfdirac}).
We denote it by ${\bf H}_{t,c}$ if $\mf h$ and $W$ are clear, and simply by $\bf H$ if both $t$ and $c$ are also fixed.

\smallskip
The purpose of this paper is to study  the $\mf{h}^*$-cohomology $H^{\bullet}(\mf{h}^*,M) := \oplus_i H^i(\mf{h}^*,M)$ and the $\mf{h}$-homology
$H_{\bullet}(\mf{h},M) := \oplus_i H_i(\mf{h},M)$ (see Section \ref{sec:halfdirac} for the definitions) as $W$-modules.  They are
naturally associated with cohomology defined by the half-Dirac operators $D_x$ and $D_y$ with twist
of a genuine character $\chi$ of $\widetilde{W}$ (see Proposition \ref{prop:dpartial}).
We show that they have nice properties in analogue of semisimple Lie algebras obtained in \cite{HPR} and \cite{HPR'},
especially we prove an analogue of the Casselman-Osborne Theorem for semisimple modules and a Hodge decomposition theorem for unitarizable modules.
Moreover, we obtain some results on relations between $H_D(M)$ and Lie algebra cohomology $H^{\bullet}(\mf{u}^*,M)$
and homology $H_{\bullet}(\mf{u},M)$ in analogy of the results on category $\mc{O}^{\mf{p}}$ obtained by Xiao and the first named author \cite{HX}.

\smallskip
We now describe our main results more precisely. Most of necessary definitions are introduced
in Sections 2 and 3. We note that Theorems A to C remain true upon replacing $H^{\bullet}(\mf{h}^*,M)$ with $H_{\bullet}(\mf{h},M)$.
Our first main result is an analogue of the Casselman-Osborne Theorem for  $H^{\bullet}(\mf{h}^*,M)$ and $H_{\bullet}(\mf{h},M)$.

\smallskip
\noindent \textbf{Theorem A (Theorem \ref{cor:wonly}).} \textit{
Let $\mc{B}$ be an abelian subalgebra of ${\bf H} \otimes C(V)$ defined before the statement of Theorem \ref{cor:voganmorphism}.
Suppose $\sigma \otimes \beta$ is an isotypical component in  $H^{\bullet}(\mf{h}^*,M)$, where $\sigma$ is an irreducible $W$-module and $\beta$ is a $\mc{B}$-character. Then \textit{Vogan's morphism}
$$\zeta_d^*:\mathrm{Irr}(\widetilde{W}) \to \mathrm{Spec}\ \mc{B}$$
satisfies the following condition
$$\zeta_d^*(\sigma \otimes \chi) = \beta.$$
Here $\chi: \widetilde{W} \to \bb{C}^{\times}$ is a genuine one-dimensional character of a double cover $p: \widetilde{W} \to W$ of $W$, satisfying $\chi^2(\widetilde{w}) = {\det}_{\mf{h}^*}(p(\widetilde{w}))$.}\\

Then we study the relation between Lie algebra cohomology $H^{\bullet}(\mf{h}^*,M)$ and Dirac cohomology $H_D(M)$ introduced in \cite{C1}.
More precisely, we have the following inclusion of $\widetilde{W}$-module homomorphisms.

\smallskip
\noindent \textbf{Theorem B (Theorem \ref{thm:incprop}).} \textit{
Let $M$ be a ${\bf H}$-module so that $D^2$ acts semisimply on $M \otimes S$, where $S$ is the spinor module (Definition \ref{def:spinor}). Then there is an injective $\widetilde{W}$-module homomorphism:} $H_D(M) \hookrightarrow H^{\bullet}(\mf{h}^*, M) \otimes \chi.$

\smallskip
If $M$ is unitary, then the inclusion in Theorem B is an isomorphism. It follows from a Hodge decomposition theorem for unitary modules.

\noindent \textbf{Theorem C (Theorem \ref{thm:unitary}).} \textit{ Let $M$ be a $\ast$-unitary module in the sense of \cite{ES}. Then one has\\
(a)\ $H_D(M) = \ker D = \ker D^2$;\\
(b)\ $M \otimes S = \ker D \oplus \mathrm{im}\ D_x \oplus \mathrm{im}\ D_y$;\\
(c)\ $\ker D_x = \ker D \oplus \mathrm{im}\ D_x$, $\ker D_y = \ker D \oplus \mathrm{im}\ D_y$,\\
where $D_x$, $D_y$ are half Dirac operators defined in Definition \ref{def:halfdirac} satisfying $D = D_x + D_y$. Consequently, }
$$H_D(M) = \ker D = \ker D_x/ \mathrm{im}\ D_x \cong H^{\bullet}(\mf{h}^*,M) \otimes \chi.$$

In the last two sections we study the Lie algebra cohomology of certain types of ${\bf H}_{t,c}$-modules. Section 6 concerns about the category $\mc{O}$ of ${\bf H}_{1,c}$-modules (see \cite{GGOR}). For example, we show that if $M(\sigma) := {\bf H}_{1,c} \otimes_{S(\mf{h}) \rtimes \bb{C}[W]} \sigma$ is a standard object in category $\mc{O}$, then
$$H_D(M(\sigma)) \cong H^{\bullet}(\mf{h}^*,M(\sigma)) \otimes \chi \cong H_{\bullet}(\mf{h},M(\sigma)) \otimes \chi \cong  \sigma \otimes \chi^{-1}.$$
This is a refinement of Proposition 5.6 in \cite{C1}. Moreover, if $L(\mathrm{triv})$ is the finite-dimensional irreducible quotient of
$M(\mathrm{triv})$ given in Proposition 1.3 of \cite{BEG}, then
$$H_D(L(\mathrm{triv})) \cong H^{\bullet}(\mf{h}^*,L(\mathrm{triv})) \otimes \chi \cong H_{\bullet}(\mf{h},L(\mathrm{triv})) \otimes \chi \cong \wedge^{\bullet}\mf{h} \otimes \chi^{-1}.$$
In other words, the inclusion in Theorem B is also an isomorphism for all standard modules and finite-dimensional modules in category $\mathcal{O}$. \\

Section 7 deals with the baby Verma modules for ${\bf H}_{0,c}$ (see \cite{G2}). We show that
if $\overline{L}(\sigma)$ be the irreducible head of a baby Verma module such that as $W$-modules,
$$H^{\bullet}(\mf{h}^*, \overline{L}(\sigma)) \cong \bigoplus_{i=1}^{k} \nu_i$$
with all $\nu_i$ being irreducible, then the whole set $\{ \nu_i \otimes {\det}_{\mf{h}^*} \ |\ i = 1, \dots, k \}$
is contained in the same Calogero-Moser cell (\cite{G2} Section 7). Combining with Theorem B, this gives an alternative proof of Corollary 5.10 in \cite{C1}.

\section{Preliminaries} \label{sec:halfdirac}
We begin this section by recalling the definition of rational Cherednik algebras ${\bf H}_{t,c}$ given in \cite{EG} and \cite{GGOR}.

\begin{definition}\label{def:cherednik}
Let $W$ be a complex reflection group acting on a complex vector space $\mf{h}$, i.e.,
$W$ is a group generated by the pseudo-reflections $s \in \mc{R}$ fixing a hyperplane $H_s \in \mf{h}$.
Let $\alpha_s \in \mf{h}^*$ be a non-zero vector so that the $W$-invariant symmetric pairing $\langle\ ,\ \rangle$
between $\mf{h}$ and $\mf{h}^*$ gives $\langle y,\alpha_s \rangle  = 0$ for all $y \in H_s$. Similarly,
we can define $\alpha_s^{\vee} \in \mf{h}$ corresponding to the action of $s$ on $\mf{h}^*$.
Set $V = \mf{h} \oplus \mf{h}^*$. The rational Cherednik algebra ${\bf H}_{t,c}$ associated to $\mf{h}$, $W$,
with parameters $t \in \bb{C}$ and $W$-invariant functions $c:\mc{R} \to \bb{C}$ is defined as the quotient of $S(V) \rtimes \bb{C}[W]$ by the relation
$$[y,x] = t\langle y, x \rangle - \sum_{s \in \mc{R}} c(s) \frac{\langle y, \alpha_s \rangle \langle \alpha_s^{\vee}, x \rangle}{\langle \alpha_s^{\vee}, \alpha_s \rangle} s$$
for all $y \in \mf{h}$ and $x \in \mf{h}^*$.
\end{definition}

Let ${\bf H}_{t,c}$ be the rational Cherednik algebra corresponding to $W, \mf{h}$ with parameters $t$ and $c$. Let $\{y_1, \dots, y_n\}$ be a basis of $\mf{h}$, and $\{x_1, \dots, x_n\}$ be the corresponding dual basis of $\mf{h}^*$. In \cite{GGOR}, a Casimir-type element $\mathbf{h}$ is defined by
$$\mathbf{h} := \sum_i (x_iy_i + y_ix_i) = 2\sum_i x_iy_i + nt - \sum_{s \in \mc{R}}c(s)s \in {\bf H}_{t,c}^W.$$
It gives a natural grading on ${\bf H}_{t,c}$ when $t = 1$. Note that the definition of $\mathbf{h}$ does not depend on our choice of basis. Following \cite{GGOR},
we make a shift and define $\Omega_{{\bf H}_{t,c}}$ in the following.
\begin{definition}\label{eq:euler}
$$\Omega_{{\bf H}_{t,c}} :=  \mathbf{h} - \sum_{s \in \mc{R}} c(s) \frac{1+\lambda_s }{1 - \lambda_s}s = 2\sum_i x_iy_i + nt - \sum_{s \in \mc{R}}\frac{2c(s)}{1 - \lambda_s} s \in {\bf H}_{t,c}^W,$$
where $\lambda_s = {\det}_{\mf{h}}(s) \in \bb{C}^{\times}$.
\end{definition}

We now define the Lie algebra (co)homology of a ${\bf H}_{t,c}$-module, analogous to the case of the $\mf{n}$ and $\overline{\mf{n}}$-(co)homology of a $\mf{g}$-module studied in \cite{K}.
\begin{definition}
Let $M$ be an ${\bf H}_{t,c}$-module. The {\bf $p^{th}$} {\bf $\mf{h}^*$-cohomology group} $H^p(\mf{h}^*, M)$ of $M$ 
arises as the $p^{th}$ derived functor of the covariant, left exact functor
$$M\mapsto M^{\mf{h}^*}=H^0((\mf{h}^*,M)$$
and can be identified with the $p^{th}$ cohomology group of the cochain complex
\begin{align*}
0 \to Hom_{\mf{h}^*}(\wedge^0\mf{h}^*,M) \xrightarrow{d_0} Hom_{\mf{h}^*}(\wedge^1\mf{h}^*,M) \xrightarrow{d_1}  \dots  \xrightarrow{d_{n-1}} Hom_{\mf{h}^*}(\wedge^n\mf{h}^*,M) \to 0,
\end{align*}
where the differential is defined by
$$d_pf(x_{i_1} \wedge \dots \wedge x_{i_p}) := \sum_j (-1)^j x_{i_j} \cdot f(x_{i_1} \wedge \dots \wedge \widehat{x_{i_l}} \wedge \dots \wedge x_{i_p}).$$
If we identify $Hom_{\mf{h}^*}(\wedge^p\mf{h}^*,M)$ with $M \otimes \wedge^p \mf{h}$,  then $H^p(\mf{h}^*,M)$ is also identified with the $p^{th}$ cohomology of the complex
\begin{align} \label{def:d}
0 \to M \otimes \wedge^0 \mf{h} \xrightarrow{d_0} M \otimes \wedge^1 \mf{h} \xrightarrow{d_1}  \dots  \xrightarrow{d_{n-1}} M \otimes \wedge^n \mf{h} \to 0,
\end{align}
where the differential is defined by
$$d_p(m \otimes y_{i_1} \wedge \dots \wedge y_{i_p}) := \sum_j x_j \cdot m \otimes y_j \wedge y_{i_1} \wedge  \dots \wedge y_{i_p}.$$
\end{definition}
The boundary map $d_p$ is well-defined because of the following proposition.

\begin{proposition} \label{prop:hstarcohomology} \mbox{}\\
(a) The definition of $d_p$ is independent of the choice of basis of $\mf{h}$ and $\mf{h}^*$.\\
(b) Treating $M$ and $\wedge^{p} \mf{h}$ as $W$-modules, then $d_p$ commutes with the action of $W$ on $M \otimes \wedge^p \mf{h}$.\\
(c) $d_{p+1} d_p = 0$.\\
Consequently, $H^p(\mf{h}^*,M)$ is a $W$-module for all $p$.
\end{proposition}
\begin{proof}
(a)\ Suppose we have another basis $y_i'$ of $\mf{h}$ given by $y_i' = \sum_j A_{ji}y_j$ for some invertible $A = (A_{ij})$. Then the corresponding dual basis $x_i'$ of $\mf{h}^*$ must be given by $x_i' = \sum_k B_{ki}x_i$, where $B = (A^{-1})^T$. Then it is easy to check that
$$\sum_j x_j' \cdot m \otimes y_j' \wedge y_{i_1} \wedge  \dots \wedge y_{i_p} = \sum_j x_j \cdot m \otimes y_j \wedge y_{i_1} \wedge  \dots \wedge y_{i_p}.$$
Hence, $d_p$ is independent of the basis of $\mf{h}$.\\
(b) For all $w \in W$,
\begin{align*}
w\cdot d_p (m \otimes y_{i_1} \wedge  \dots \wedge y_{i_p}) &= w \sum_j x_j \cdot m \otimes y_j \wedge y_{i_1} \wedge  \dots \wedge y_{i_p}\\
&= \sum_j wx_j \cdot m \otimes w(y_j) \wedge w(y_{i_1}) \wedge  \dots \wedge w(y_{i_p})\\
&= \sum_j w(x_j)w \cdot m \otimes w(y_j) \wedge w(y_{i_1}) \wedge  \dots \wedge w(y_{i_p})\\
&= d_p (w \cdot m \otimes w(y_{i_1}) \wedge  \dots \wedge w(y_{i_p}))\\
&= d_p (w \cdot (m \otimes y_{i_1} \wedge  \dots \wedge y_{i_p})),
\end{align*}
where the third equality follows from the definition of ${\bf H}_{t,c}$, and fourth equality follows from (a).\\
(c) By the definition of $d_p$,
\begin{align*}
d_{p+1}d_p(m \otimes y_{i_1} \wedge \dots \wedge y_{i_p}) &= \sum_{i,j} x_i x_j \cdot m \otimes y_i \wedge y_j \wedge y_{i_1} \wedge  \dots \wedge y_{i_p}\\
&= \sum_{i<j} x_i x_j \cdot m \otimes (y_i \wedge y_j + y_j \wedge y_i) \wedge y_{i_1} \wedge  \dots \wedge y_{i_p}\\
& \ \ + \sum_{i} x_i^2 \cdot m \otimes (y_i \wedge y_i) \wedge y_{i_1} \wedge  \dots \wedge y_{i_p},
\end{align*}
which is equal to $0$ since $y_i \wedge y_j = - y_j \wedge y_i$ for all $i, j$.
\end{proof}

\begin{definition}
Let $M$ be an ${\bf H}_{t,c}$-module. The {\bf $p^{th}$} {\bf $\mf{h}$-homology group} $H_p(\mf{h}, M)$ of $M$ arises as the $p^{th}$  
derived functor of the covaraint, right exact functor 
$$M\mapsto M/\mf{h}M=H_0(\mf{h},M)$$ 
on the category of $\mf{h}$-modules.  It can be calculated as the $p^{th}$ homology group of the chain complex
\begin{align} \label{def:partial}
0 \to M \otimes \wedge^n \mf{h} \xrightarrow{\partial_n} M \otimes \wedge^{n-1} \mf{h} \xrightarrow{\partial_{n-1}}  \dots \xrightarrow{\partial_2} M \otimes \mf{h} \xrightarrow{\partial_1} M \to 0,
\end{align}
where the differential is defined by
\begin{align*}
\partial_p(m \otimes y_{i_1} \wedge \dots \wedge y_{i_p}) &:= \sum_k (-1)^k y_{i_k} \cdot m \otimes y_{i_1} \wedge \dots \wedge \widehat{y_{i_k}} \wedge \dots \wedge y_{i_p}\\
&= \sum_k \sum_j (-1)^k \langle x_j ,y_{i_k} \rangle y_j \cdot m \otimes  y_{i_1} \wedge \dots \wedge \widehat{y_{i_k}} \wedge \dots \wedge y_{i_p}.
\end{align*}
\end{definition}
We have the following proposition whose proof is similar to that of $\mf{h}^*$-cohomology.
\begin{proposition} \mbox{}\\
(a) The definition of $\partial_p$ is independent of the choice of basis of $\mf{h}$ and $\mf{h}^*$.\\
(b) Treating $M$ and $\wedge^{p} \mf{h}$ as $W$-modules, then $\partial_p$ commutes with the action of $W$ on $M \otimes \wedge^{p} \mf{h}$.\\
(c) $\partial_{p-1} \partial_p = 0$.\\
Consequently, $H_p(\mf{h},M)$ is a $W$-module.
\end{proposition}

From now on, we write $H^{\bullet}(\mf{h}^*,M) := \oplus_i H^{i}(\mf{h}^*,M)$ and $H_{\bullet}(\mf{h},M) := \oplus_i H_{i}(\mf{h},M)$ as the ungraded sum of the (co)homology spaces. Similarly, we can define $\mf{h}^*$-homology $H_{\bullet}(\mf{h}^*,M)$ and $\mf{h}$-cohomology $H^{\bullet}(\mf{h},M)$. They are related to $\mf{h}^*$-cohomology and $\mf{h}$-homology respectively by the Poincar\'{e} duality:
\begin{proposition}[Poincar\'{e} duality]\label{poincare}
Let $\mf{c} = \mf{h}$ or $\mf{h}^*$, then the perfect pairing $\wedge^p \mf{c} \times \wedge^{n-p} \mf{c} \to \wedge^n \mf{c}$ gives a Poincar\'{e} duality between the $H^{i}(\mf{c}, M)$ and the homology $H_{n -i}(\mf{c}, M)$ defined above. More precisely, for all $i = 1, \dots n$,
$$H_{i}(\mf{c}, M) \cong H^{n - i}(\mf{c}, M) \otimes \wedge^n \mf{c},\ \ \ \text{or}\ \ H^{i}(\mf{c}, M) \cong H_{n - i}(\mf{c}, M) \otimes \wedge^n \mf{c}^*$$
as $W$-modules.
\end{proposition}

We now recall the Clifford algebra and spinor module of $V = \mf{h} \oplus \mf{h}^*$. Define a $W$-invariant bilinear product on $V$ by $\langle\ ,\ \rangle$ by $\langle x_i, x_j \rangle = \langle y_i , y_j \rangle = 0$, $\langle x_i , y_j\rangle = \delta_{ij}$ (this is the same as the pairing given in Definition \ref{def:cherednik}). The \textbf{Clifford algebra} $C(V)$ with respect to $\langle\ ,\ \rangle$ is the tensor algebra of $V$ subject to the relations
$$x_ix_j + x_jx_i = y_iy_j + y_jy_i = 0,\ x_iy_j + y_jx_i = -2\delta_{ij}.$$
There is a natural injection $\widetilde{W} \hookrightarrow \mathrm{Pin}(V) \hookrightarrow C(V)^{\times}$, where $p: \widetilde{W} \to W$ is the double cover of $W$ given by the pull-back of the Pin cover $p: \mathrm{Pin}(V) \to O(V)$. For any $s \in \mc{R}$, let
\begin{align} \label{taus}
\mu_s = \frac{\sqrt{\lambda_s} - 1/\sqrt{\lambda_s}}{2\langle \alpha_s^{\vee}, \alpha_s\rangle} \alpha_s \alpha_s^{\vee} + \sqrt{\lambda_s} \in C(V),
\end{align}
where $\sqrt{\lambda_s}$ is a choice of the square root of $\lambda_s = \det_{\mf{h}}(s)$. Then the calculations in Lemma 4.6 of \cite{C1} shows that $\{ \pm \mu_s \} = p^{-1}(s) \subset \widetilde{W}$, and $\{ \pm \mu_s | s \in \mc{R} \}$ generate $\widetilde{W}$.\\

\begin{definition} \label{def:spinor}
The {\bf spinor module} $S$ corresponding to the Clifford algebra $C(V)$ can be realized as $S \cong \wedge^{\bullet} \mf{h}$ as vector spaces. The action of $C(V)$  on $S$ is defined by
\begin{align*}
x(y_{i_1} \wedge \dots \wedge y_{i_p}) &= 2 \sum_j(-1)^j \langle x,y_{i_j} \rangle y_{i_1} \wedge \dots \wedge \widehat{y_{i_j}} \wedge \dots \wedge y_{i_p}, \ \ x \in \mf{h}^*;\\
y(y_{i_1} \wedge \dots \wedge y_{i_p}) &= y \wedge y_{i_1} \wedge \dots \wedge y_{i_p}, \ \ y \in \mf{h}.
\end{align*}
\end{definition}
The following proposition describes $S$ as a $\widetilde{W}$-module:
\begin{proposition} \label{prop:hstar}
Every $\widetilde{w} \in \widetilde{W}$ preserves every $\wedge^l \mf{h} \subset S$. In particular, for every $v_i \in \mf{h}$,
$$\widetilde{w} \cdot v_1 \wedge v_2 \wedge \dots \wedge v_l = \chi(\widetilde{w}) \cdot (p(\widetilde{w})(v_1) \wedge p(\widetilde{w})(v_2) \wedge \dots \wedge p(\widetilde{w})(v_l))$$
with $\chi$ being a genuine one-dimensional $\widetilde{W}$-module satisfying $\chi^2(\widetilde{w}) = {\det}_{\mf{h}^*}(p(\widetilde{w}))$. Therefore, as $\widetilde{W}$-modules,
\begin{align} \label{eq:spinor}
S \cong \chi \otimes \wedge^{\bullet}\mf{h},
\end{align}
where $\wedge^{\bullet}\mf{h}$ is considered as a $\widetilde{W}$-module that factors through $p: \widetilde{W} \to W$.
\end{proposition}
\begin{proof}
We only need to prove the proposition for all generators $\pm \mu_s \in \widetilde{W}$.
Let $\alpha_s^{\vee} \in \mf{h}$, $\alpha_s \in \mf{h}^*$ be as in Definition \ref{def:cherednik},
and fix a basis $\{y_1, \dots, y_n\}$ of $\mf{h}$ with $y_1 = \alpha_s^{\vee}$, $y_2, \dots, y_n \in \ker (Id - s)|_{\mf{h}}$. Then
$$s(y_1) = \lambda_s y_1,\ \ \ \ \ s(y_i) = y_i \text{ for }\ i >1.$$
For $1 \leq l_1 < \dots < l_k \leq n$, it is easy to check that
$$\alpha_s \alpha_s^{\vee} \cdot y_{l_1} \wedge \dots \wedge y_{l_k} =
\begin{cases}
    0,& \text{if } l_1 = 1;\\
    -2\langle \alpha_s^{\vee}, \alpha_s\rangle y_{l_1} \wedge \dots \wedge y_{l_k},      & \text{otherwise}.
\end{cases} $$
Using \eqref{taus}, $\pm \mu_s$ acts on $y_{l_1} \wedge \dots \wedge y_{l_k}$ by $\pm \mu_s \cdot y_{l_1} \wedge \dots \wedge y_{l_k}
= \frac{\pm 1}{\sqrt{\lambda_s}} s(y_{l_1}) \wedge \dots \wedge s(y_{l_k})$ in both cases. Hence the result follows from the fact that $(\frac{\pm 1}{\sqrt{\lambda_s}})^2 = \lambda_s^{-1} = {\det}_{\mf{h}^*}(s)$.
\end{proof}

\begin{definition}\label{def:halfdirac}
Let $D_x$, $D_y$ be elements  in ${\bf H} \otimes C(V)$ given by
$$D_x := \sum_i x_i \otimes y_i; \ \ \ \ \ D_y:= \sum_i y_i \otimes x_i.$$
If $M$ is a ${\bf H}_{t,c}$-module, then $D_x$, $D_y$ acts on $M \otimes S$ by
$$D_x(m \otimes s) =  \sum_{i} x_i\cdot m \otimes y_is;\ \ \ \ \ D_y(m \otimes s) =  \sum_{i} y_i\cdot m \otimes x_is.$$
\end{definition}

\begin{theorem}[\cite{C1} Lemma 2.4, Proposition 4.9] \label{thm:dsquare}   \mbox{} We have \\
(a)\ $D_x$ and $D_y$ are independent of the choice of basis of $\mf{h}$.\\
(b)\ Let $\Delta: \bb{C}[\widetilde{W}] \to {\bf H}_{t,c} \otimes C(V)$ be the diagonal map $\widetilde{w} \mapsto p(\widetilde{w}) \otimes \widetilde{w}$. Then $D_x$ and $D_y$ commute with $\Delta(\bb{C}[\widetilde{W}])$. \\
(c)\ $D_x^2 = D_y^2 = 0$.\\
Consequently, $\ker D_x/\mathrm{im}\ D_x$ and $\ker D_y/\mathrm{im}\ D_y$ are naturally $\widetilde{W}$-modules.
\end{theorem}

Recall from Proposition \ref{prop:hstar} that we can identify $S$ with $\wedge^{\bullet}\mf{h} \otimes \chi$ as $\widetilde{W}$-modules. With this identification,
$$D_x=d\otimes \chi; \ \ \ D_y=2\partial \otimes \chi.$$
Thus, we have the following proposition relating the above operators with the Lie algebra (co)homology.

\begin{proposition} \label{prop:dpartial}
There are $\widetilde{W}$-module isomorphisms:
$$\ker D_x/\mathrm{im}\ D_x  \cong H^{\bullet}(\mf{h}^*,M) \otimes \chi \ \ \text{and} \ \  \ker D_y/\mathrm{im}\ D_y \cong H_{\bullet}(\mf{h},M) \otimes \chi.$$
\end{proposition}

\section{An analogue of Casselman-Osborne Theorem} \label{sec:voganconjecture}
In this section, we prove a version of Casselman-Osborne Theorem for Lie algebra cohomology
of ${\bf H}_{t,c}$ and associated with $D_x$ and $D_y$.
It relates the `central character' of a ${\bf H}_{t,c}$-module $M$ (denoted as $\mc{B}$-character below) to
the central characters of $\bb{C}[W]$-modules $H^{\bullet}(\mf{h}^*,M)$ and $H_{\bullet}(\mf{h},M)$.

We set ${\bf H}={\bf H}_{t,c}$.
Recall that we have the following Poincar\'{e}-Birkhoff-Witt decomposition ${\bf H} \cong S(\mf{h}) \otimes \bb{C}[W] \otimes S(\mf{h}^*)$ as vector spaces. We define a $\bb{C}^*$-action on ${\bf H}$ by
$$\lambda \cdot x = \lambda^{-1} x\ , \ \lambda \cdot y = \lambda y\ ,\  \lambda \cdot w = w$$
for all $x \in \mf{h}^*$, $y \in \mf{h}$ and $w \in W$.  We also define a $\bb{C}^*$-action on $C(V) \cong \wedge^{\bullet} \mf{h} \otimes \wedge^{\bullet} \mf{h}^*$ by
$$\lambda \cdot x = \lambda^{-1} x\ , \ \lambda \cdot y = \lambda y$$
for all $x \in \mf{h}^*$, $y \in \mf{h}$.

\begin{definition}\label{def:cstar} \mbox{}We define a subalgebra $\mathrm{\mathbf{A}}$ of ${\bf H}\otimes C(V)$ by setting
$$\mathrm{\mathbf{A}} := ({\bf H} \otimes C(V))^{\bb{C}^*}.$$
\end{definition}

It is easy to check that we have the following  $\Delta(\bb{C}[\widetilde{W}])$-module isomorphism
$$\mathrm{\mathbf{A}} \cong \bigoplus_{k_1 + l_1 = k_2 + l_2} (S^{k_1}(\mf{h}) \otimes \bb{C}[W] \otimes S^{k_2}(\mf{h}^*)) \otimes (\wedge^{l_1} \mf{h} \otimes \wedge^{l_2} \mf{h}^*) \subset {\bf H} \otimes C(V).$$

We have a filtration ${\bf H}^0 \subset {\bf H}^1 \subset {\bf H}^2 \subset \dots$ on ${\bf H}$ by taking $\deg(x) = \deg(y) = 1$ for $x \in \mf{h}^*$, $y \in \mf{h}$ and $\deg(w) =0$ for all $w \in W$. Then the graded algebra is given by $gr({\bf H}) \cong {\bf H}_{0,0} \cong S(V) \rtimes \bb{C}[W]$. With the filtration on ${\bf H}$, define the filtration $\mathrm{\mathbf{A}}^{0} \subset \mathrm{\mathbf{A}}^{1} \subset \dots$ of $\mathrm{\mathbf{A}}$, where $\mathrm{\mathbf{A}}^{n} = \mathrm{\mathbf{A}} \cap {\bf H}^n \otimes C(V)$. So we have a graded algebra
$$gr(\mathrm{\mathbf{A}}) \cong \bigoplus_{k_1 + l_1 = k_2 + l_2} (S^{k_1}(\mf{h}) \otimes S^{k_2}(\mf{h}^*))\rtimes \bb{C}[W] \otimes (\wedge^{l_1} \mf{h} \otimes \wedge^{l_2} \mf{h}^*)\ \ \ \subset\ \ \  {\bf H}_{0,0} \otimes C(V).$$


By definition, we have $\Delta(\bb{C}[\widetilde{W}]), D_x, D_y$ are all contained in $\mathbf{A}$.
Therefore, $\widetilde{W}$ acts on $\mathrm{\mathbf{A}}$ by conjugation.
We denote by $\mathrm{\mathbf{A}}^{\widetilde{W}}$ the subalgebra of $\widetilde{W}$-invariants in $\mathrm{\mathbf{A}}$. By Proposition \ref{prop:dpartial}(b), $D_x$, $D_y \in \mathrm{\mathbf{A}}^{\widetilde{W}}$. Define the maps $\delta_d, \delta_{\partial}:  \mathrm{\mathbf{A}}^{\widetilde{W}} \to \mathrm{\mathbf{A}}^{\widetilde{W}}$ by
\begin{align}\label{eq:delta}
\delta_da = D_xa - \epsilon(a)D_x,\ \ \ \ \ \delta_{\partial}a = D_y a - \epsilon(a)D_y,
\end{align}
where $\epsilon(a) = a$ if $a \in {\bf H} \otimes C^{even}(V)$ and $\epsilon(a) = -a$ if $a \in {\bf H} \otimes C^{odd}(V)$.


The main theorem of this section is the following:
\begin{theorem}\label{thm:voganmorphism}
For $\delta_d, \delta_{\partial}:  \mathrm{\mathbf{A}}^{\widetilde{W}} \to \mathrm{\mathbf{A}}^{\widetilde{W}}$ defined in \eqref{eq:delta}, we have
$$\ker \delta_{d} = \mathrm{im}\ \delta_{d} \oplus  \Delta(\bb{C}[\widetilde{W}]^{\widetilde{W}}).$$
The results hold analogously for $\delta_{\partial}$.
\end{theorem}

For the rest of this section, we will only prove the theorem for $\delta_d$. The proof of $\delta_{\partial}$ is analogous to that of $\delta_d$.

\begin{lemma}\label{lem:subalg} \mbox{} We have\\
(a)\ $\delta_d^2 = 0$ and consequently $\mathrm{im}\ \delta_d \subset \ker \delta_d$.\\
(b)\ The map $\delta_d$ is an odd derivation, i.e. $\delta_d(ab) = \delta_d(a)b + \epsilon(a) \delta_d(b)$. Therefore, if $a, b \in \ker \delta_d$, $ab \in \ker \delta_d$ and $\ker \delta_d$ is a subalgebra of $\mathrm{\mathbf{A}}^{\widetilde{W}}$.\\
\end{lemma}
\begin{proof}\mbox{} (a) By the definition of $\delta_d$, $\delta_d^2a = D_x^2 a -aD_x^2 = 0- 0 = 0$.\\
(b) $\delta_d(ab) = D_x(ab) - \epsilon(ab)D_x = D_xab - \epsilon(a)D_xb + \epsilon(a)D_xb - \epsilon(a)\epsilon(b)D_x = \delta_d(a)b + \epsilon(a) \delta_d(b)$.
\end{proof}

Note that the action $\delta_d$ increases the filtration by $1$, while the action of $\Delta(\widetilde{w})$ preserves the filtration. So $gr({\mathrm{\mathbf{A}}})^{\widetilde{W}} = gr(\mathrm{\mathbf{A}}^{\widetilde{W}})$, and $\delta_d$ descend to the map
$$\overline{\delta_d}: gr(\mathrm{\mathbf{A}})^{\widetilde{W}} \to gr(\mathrm{\mathbf{A}})^{\widetilde{W}}.$$

\begin{lemma} \label{lem:voganinclusion}  Let $\overline{\delta_d}: gr(\mathrm{\mathbf{A}})^{\widetilde{W}} \to gr(\mathrm{\mathbf{A}})^{\widetilde{W}}$ be defined as in Equation \eqref{eq:delta}. Then the following holds: \\
(a)\ $\Delta(\bb{C}[\widetilde{W}]^{\widetilde{W}})  \subset \ker \overline{\delta_d}$.\\
(b)\ $\mathrm{im}\ \overline{\delta_d} \cap \Delta(\bb{C}[\widetilde{W}]^{\widetilde{W}}) = 0$, hence
$$\mathrm{im}\ \overline{\delta_d}\ \oplus\ \Delta(\bb{C}[\widetilde{W}]^{\widetilde{W}}) \subset \ker \overline{\delta_d}.$$
\end{lemma}
\begin{proof}
(a) For all $\widetilde{w} \in \widetilde{W}$, Theorem \ref{thm:dsquare}(b) says that $\Delta(\widetilde{w})D_x - D_x\Delta(\widetilde{w}) = 0$. Also, recall $\widetilde{w}$ is generated by $\pm \mu_s \in C^{even}(V)$. Hence $\epsilon(\Delta(\widetilde{w})) = \Delta(\widetilde{w})$ and $\Delta(\bb{C}[\widetilde{W}]^{\widetilde{W}}) \subset \ker \overline{\delta_d}$.\\

(b) We have already seen that $\overline{\delta_d}^2 = 0$. So $\mathrm{im}\ \overline{\delta_d} \subset \ker \overline{\delta_d}$. On the other hand, every summand in $\mathrm{im}\ \delta_d$ must have an $x_i \in \mf{h}^*$ in its $S(V)$-component, while every element in $\Delta(\bb{C}[\widetilde{W}]^{\widetilde{W}})$ does not contain any $\mf{h}^*$ factor. Hence they must be mutually disjoint.
\end{proof}

\begin{proposition}\label{prop:bardeltadecomp}  We have
$$\ker \overline{\delta_d} = \mathrm{im}\ \overline{\delta_d}\ \oplus\  \Delta(\bb{C}[\widetilde{W}]^{\widetilde{W}}).$$
\end{proposition}
\begin{proof}
By the above lemma, we just need to show
$$\ker \overline{\delta_d} \subset \mathrm{im}\ \overline{\delta_d}\ \oplus\ \Delta(\bb{C}[\widetilde{W}]^{\widetilde{W}}).$$
Suppose $a \in \ker \overline{\delta_d}$, write $a = a_1 + \dots + a_{|W|} \in gr(\mathrm{\mathbf{A}})^{\widetilde{W}}$ where $a_i$ is the sum of the elements of the form $fw_i \otimes g$ for $f \in S^{k_1}(\mf{h}) \otimes S^{k_2}(\mf{h}^*)$, $w_i \in W$ and $g \in \wedge^{l_1} \mf{h} \otimes \wedge^{l_2}\mf{h}^*$ with $k_1 + l_1 = k_2 + l_2$.\\

For simplicity of notation, we write $\overline{\delta_d}(\alpha) = D_x\alpha - \epsilon(\alpha) D_x$ for all elements $\alpha \in gr(\mathrm{\mathbf{A}})$. Then $\overline{\delta_d}(a) = \overline{\delta_d}(a_1) + \dots + \overline{\delta_d}(a_{|W|}) = 0$. Note that $\overline{\delta_d}$ does not change the $w_i$ component on each $a_i$. Hence, we must have
$$\overline{\delta_d}(a_i) = 0, \ \ \text{for all} \ i.$$
The summands in $a_i$ are of the form
$$fw_i \otimes g = (f \otimes g\cdot \widetilde{w_i}^{-1})(w_i \otimes \widetilde{w_i}) = [(fw_i \otimes g )\cdot \Delta (\widetilde{w_i}^{-1})]\Delta( \widetilde{w_i})$$
for some $\widetilde{w_i} \in \widetilde{W}$ satisfying $p(\widetilde{w_i}) = w_i$. Now $\overline{\delta_d}(a_i) = 0$ means
$$0 = \overline{\delta_d}[a_i \cdot \Delta(\widetilde{w_i}^{-1})\cdot \Delta(\widetilde{w_i})] = \overline{\delta_d}(a_i \cdot \Delta(\widetilde{w_i}^{-1}))\cdot \Delta(\widetilde{w_i}) \pm (a_i \cdot \Delta(\widetilde{w_i}^{-1})) \cdot \overline{\delta_d}(\Delta( \widetilde{w_i})),$$
but we know $\overline{\delta_d}(\Delta(\widetilde{w_i})) = 0$ since $D_x$ commutes with $\Delta(\bb{C}[\widetilde{W}])$, hence
$$\overline{\delta_d}(a_i \Delta(\widetilde{w_i}^{-1})) = 0,\ \ \ \ a_i \Delta(\widetilde{w_i}^{-1}) \in (S^{k_1}(\mf{h}) \otimes S^{k_2}(\mf{h}^*)) \otimes (\wedge^{l_1} \mf{h} \otimes \wedge^{l_2}\mf{h}^*),$$
with $k_1 + l_1 = k_2 + l_2$.\\

It follows that $\overline{\delta_d}$ is  the differential in the Koszul complex
$$\bigoplus_m \bigoplus_{k_2 + l_2 = m} S^{k_2}(\mf{h}^*) \otimes \wedge^{l_2} \mf{h}^*,$$
which has cohomology $\bb{C}$ on degree $m = 0$ and zero at other degrees. Therefore, byupon restricting $\overline{\delta_d}$ to
$$\bigoplus_m \bigoplus_{k_1 + l_1 = k_2 + l_2 = m} S^{k_1}(\mf{h}) \otimes \wedge^{l_1} \mf{h}^*  \otimes S^{k_2}(\mf{h}^*) \otimes \wedge^{l_2} \mf{h}^*,$$ we have
$$\ker \overline{\delta_d}= \mathrm{im}\ \overline{\delta_d} \oplus \bigoplus_{k_1 + l_1 = 0} S^{k_1}(\mf{h}) \otimes \wedge^{l_1}(\mf{h}) \otimes \bb{C} (1 \otimes 1) = \mathrm{im}\ \overline{\delta_d} \oplus \bb{C}(1 \otimes 1).$$
In particular, we have
$$a_i \Delta(\widetilde{w_i}^{-1}) = \overline{\delta_d} z_i + \beta_i(1 \otimes 1),$$
where $\beta_i \in \bb{C}$ and $a_i = \overline{\delta_d}(z_i \Delta(\widetilde{w_i})) + \beta_i \Delta(\widetilde{w_i})$. Hence,
$$a = \sum_i a_i = \overline{\delta_d}(\sum_i z_i \Delta(\widetilde{w_i})) + \sum_i \beta_i \Delta(\widetilde{w_i}).$$
Therefore, $\sum_i \beta_i \Delta(\widetilde{w_i})$ must be $\widetilde{W}$-invariant, i.e. $\sum_i \beta_i \Delta(\widetilde{w_i}) \in  \Delta(\bb{C}[\widetilde{W}]^{\widetilde{W}})$ and the proposition is proved.
\end{proof}

To finish the proof of Theorem \ref{thm:voganmorphism}, one needs to remove the $^{-}$'s in Proposition \ref{prop:bardeltadecomp}:\\

\noindent \textit{Proof of Theorem \ref{thm:voganmorphism}.} We have already shown in Lemma \ref{lem:voganinclusion} that $\mathrm{im}\ \delta_d \oplus \bb{C}[\widetilde{W}]^{\widetilde{W}} \subset \ker \delta_d$ so we check the reverse inclusion. Suppose $a \in \mathrm{\mathbf{A}}^{\widetilde{W},n}$ is in $\ker \delta_d$, then $\overline{a} \in \ker \overline{\delta_d}$ and $\overline{\delta_d}(\overline{a}) = 0$. By Proposition \ref{prop:bardeltadecomp}, there exists $\overline{b} \in gr(\mathrm{\mathbf{A}}^{\widetilde{W}})_{n-1}$ and $s \in \bb{C}[\widetilde{W}]^{\widetilde{W}}$ such that
$$\overline{a} = \overline{\delta_d}\bar{b} + \Delta(s).$$
Pick $b \in \mathrm{\mathbf{A}}^{\widetilde{W},n-1}$ such that $gr(b) = \overline{b}$.  Then
$$\overline{a - \delta_d b - \Delta(s)} = 0$$
and hence $a - \delta_d b - \Delta(s) \in \mathrm{\mathbf{A}}^{\widetilde{W},n-1}$. Note that
$$a - \delta_d b - \Delta(s) \in \ker \delta_d.$$
By induction on $n$, we have $a - \delta_d b - \Delta(s) \in \mathrm{im}\ \delta_d \oplus \bb{C}[\widetilde{W}]^{\widetilde{W}}$. Thus, $a \in \mathrm{im}\ \delta_d \oplus \bb{C}[\widetilde{W}]^{\widetilde{W}}$ as required.
\qed \\

Let $\mc{B}$ be an abelian subalgebra of $\ker \delta_{d}\cap ({\bf H} \otimes C^{even}(V))$. Then, by definition of $\delta_d$, $\mc{B}$ commutes with $D_x$ and $\Delta(\bb{C}[\widetilde{W}])$. So $\ker D_x/ \mathrm{im}\ D_x$ is naturally a $\bb{C}[\widetilde{W}] \otimes \mc{B}$-module.
\begin{theorem}\label{cor:voganmorphism} Let $\zeta_{d}: \mc{B} \to \Delta(\bb{C}[\widetilde{W}]^{\widetilde{W}})$ be the restriction to $\mc{B}$ of the projection map given by Theorem \ref{thm:voganmorphism}. Then we have\\
(a) $\zeta_d$ is an algebra homomorphism.\\
(b) Suppose that $\widetilde{\sigma} \otimes \beta$ is an isotypical component
in $\ker D_x/ \mathrm{im}\ D_x$, where $\widetilde{\sigma}$ is an irreducible $\widetilde{W}$-module and $\beta$ is a $\mc{B}$-character.
Then the morphism
$$\zeta_d^*:\mathrm{Irr}(\widetilde{W}) \to \mathrm{Spec}\ \mc{B}$$ satisfies the following condition
$$\zeta_d^*(\widetilde{\sigma}) = \beta.$$
The results hold analogously by replacing $D_x$ with $D_y$.
\end{theorem}
\begin{proof}
We only present the proof for $\delta_d$ below:\\
(a) Let $b_1, b_2 \in \mc{B}$, then $b_i = \delta_d(a_i) + \zeta_d(b_i)$ for some $a_i \in \mathrm{\mathbf{A}}$, so
\begin{align*}
b_1b_2 &= \delta_d(a_1)\delta_d(a_2) + \delta_d(a_1)\zeta_d(b_2) + \zeta_d(b_1)\delta_d(a_2) + \zeta_d(b_1)\zeta_d(b_2)\\
&= \delta_d (a_1\delta_d(a_2) + a_1\zeta_d(b_2) + \zeta_d(b_1)a_2) + \zeta_d(b_1)\zeta_d(b_2).
\end{align*}
The second equality comes from the facts that $\delta_d^2 = 0$, $\delta_d(\Delta(\widetilde{w})) = 0$ and $\delta_d$ is a derivation. Hence by the uniqueness of the decomposition in Theorem \ref{thm:voganmorphism},
$$\zeta_d(b_1b_2) = \zeta_d(b_1)\zeta_d(b_2)$$
and the result follows.\\
(b)\ Let $0 \neq z \in \mc{B}$. By Theorem \ref{thm:voganmorphism}, we have
$$z = \zeta_d(z) + D_x a - \epsilon(a) D_x$$
for some $a \in {\bf H} \otimes C(V)$. Let $\widetilde{\alpha} \neq 0$ be an element in the $\widetilde{\sigma} \otimes \beta$ component of $\ker D_x/ \mathrm{im}\ D_x$.  Then
\begin{align*}
z\cdot \widetilde{\alpha} &= \zeta_d(z)\cdot \widetilde{\alpha} + D_x a \cdot \widetilde{\alpha} - \epsilon(a) D_x \cdot \widetilde{\alpha} \\
\beta(z)\widetilde{\alpha} &= \widetilde{\sigma}(\zeta_d(z))\widetilde{\alpha} + D_x a \cdot \widetilde{\alpha} \\
(\beta(z) - \widetilde{\sigma}(\zeta_d(z)))\widetilde{\alpha} &= D_x a \cdot \widetilde{\alpha},
\end{align*}
where the second equality comes from the fact that $\widetilde{\alpha}$ is an $\widetilde{\sigma} \otimes \beta$ isotypic element in $\ker D_x/\mathrm{im}\ D_x$.
Therefore, the left hand side of the equality is in $\ker D_x/ \mathrm{im}\ D_x$, and the right hand side is in $\mathrm{im}\ D_x$ and hence it must be equal to zero, i.e.
$$(\beta(z) - \widetilde{\sigma}(\zeta_d(z)))\widetilde{\alpha} = 0.$$
However,  $\widetilde{\alpha} \neq 0$, and it follows that $\beta(z) - \widetilde{\sigma}(\zeta_d(z)) = 0$ for all $z \in \mc{B}$ and the result holds.
\end{proof}
With the identification $\ker D_x/\mathrm{im}\ D_x \cong H^{\bullet}(\mf{h}^*,M) \otimes \chi$, the following result is straightforward:
\begin{theorem} \label{cor:wonly}
Retain the notations in Theorem \ref{cor:voganmorphism}. Suppose $\sigma \otimes \beta$ is an isotypical component in  $H^{\bullet}(\mf{h}^*,M)$, where $\sigma$ is a $W$-module and $\beta$ is a $\mc{B}$-character. Then
$$\zeta_d^*(\sigma \otimes \chi) = \beta.$$
The similar statement holds for $H_{\bullet}(\mf{h},M)$ and $\zeta^*_{\partial}$. (Note that $\sigma \otimes \chi$ is an irreducible $\widetilde{W}$-module for any irreducible $W$-module $\sigma$.)
\end{theorem}

\section{Embedding of Dirac cohomology into Lie algebra cohomology}

In this section, we give a criterion such that the Dirac cohomology $H_D(M)$ of $M$ can be embedded into its Lie algebra (co)homology.

\begin{definition}[\cite{C1}] \label{def:dirac}
Let $D:= D_x + D_y \in {\bf H}_{t,c} \otimes C(V)$. The Dirac cohomology of an ${\bf H}_{t,c}$-module $M$ is defined by
$$H_D(M) := \ker D/\mathrm{im}\ D,$$
where $D: M \otimes S \to M \otimes S$ is defined as in Definition \ref{def:halfdirac}.
\end{definition}

The main theorem of this section is the following:

\begin{theorem}\label{thm:incprop}
If $M$ is an ${\bf H}_{t,c}$-module so that $D^2$ acts semisimply on $M \otimes S$, then we have the
following $\widetilde{W}$-module injective homomorphisms
$$H_D(M) \hookrightarrow  H^{\bullet}(\mf{h}^*, M) \otimes \chi,\ \ \ \ \ H_D(M) \hookrightarrow  H_{\bullet}(\mf{h}, M) \otimes \chi.$$
\end{theorem}

We now prove the first inclusion of Theorem \ref{thm:incprop}. The second inclusion can be proved in an analogous way.

\begin{lemma}\label{thm:filter}
If we identify $S$ with $\wedge^{\bullet}\mf{h}$ as vector spaces (Definition \ref{def:spinor}) and define an increasing $\widetilde{W}$-invariant filtration on $\ker D$ in $M \otimes S$ by
$$0 = (\ker D)_{-1} \subset (\ker D)_0 \subset \dots \subset (\ker D)_{n-1} \subset (\ker D)_n = \ker D,$$
with $(\ker D)_i = \ker D \cap (\bigoplus_{p = 0}^i M \otimes \wedge^p \mf{h})$. Then we have an injective $W$-module homomorphism:
$$f: gr(\ker D) \hookrightarrow \ker D_x \cap \ker D^2.$$
\end{lemma}
\begin{remark}
Since the action of $\Delta(\bb{C}[\widetilde{W}])$ on $M \otimes S$ preserves $M \otimes \wedge^{l}\mf{h}$ for each $l$, we have an isomorphism of $\widetilde{W}$-modules $\ker D \cong gr(\ker D)$ and hence the injection of $\widetilde{W}$-modules in the above lemma can be rewritten as:
$$f': \ker D \hookrightarrow \ker D_x \cap \ker D^2.$$
\end{remark}
\begin{proof}
Suppose $\widetilde{m_0} + \dots + \widetilde{m_i} \in (\ker D)_i$, with $\widetilde{m_p} \in M \otimes \wedge^p \mf{h}$, then
$$0 = D(\widetilde{m_0} + \dots + \widetilde{m_i}) = (D_x + D_y)(\widetilde{m_0} + \dots + \widetilde{m_i}) = \sum_{p=0}^i D_y \widetilde{m_p} + \sum_{p=0}^{i-1} D_x \widetilde{m_p} + D_x(\widetilde{m_i}).$$
Note that the last term is the only term in $M \otimes \wedge^{i+1} \mf{h}$, hence $\widetilde{m_i} \in \ker D_x$ and we can define a map
\begin{align*}
f_i: (\ker D)_i/(\ker D)_{i+1} & \to \ker D_x;\\
\widetilde{m_0} + \dots + \widetilde{m_i} & \mapsto \widetilde{m_i}.
\end{align*}
To check the image is in $\ker D^2$, we note that $D^2 = D_xD_y + D_yD_x$ preserves the degrees. Then
$$0 = D^2(\widetilde{m_0} + \dots \widetilde{m_i}) = D^2(\widetilde{m_0}) + \dots + D^2(\widetilde{m_i}),$$
and it follows that $D^2(\widetilde{m_i}) = 0$.
Now we show that $f_i$ is injective.  If
$\widetilde{m}' = \widetilde{m_0}' + \dots + \widetilde{m_i}' \in (\ker D)_i$ such that $f_i(\widetilde{m_0}' + \dots + \widetilde{m_i}') = 0$,
then $\widetilde{m_i}' = 0$. Hence, $\widetilde{m}' \in (\ker D)_{i-1}$. Thus,
 we have an injective map
$$f := \bigoplus_i f_i : gr(\ker D) \hookrightarrow \ker D_x \cap \ker D^2.$$
\end{proof}

\noindent \textit{Proof of Theorem \ref{thm:incprop}}. By hypothesis, we can decompose $M \otimes S$ into $M \otimes S = \ker D^2\ \oplus\ \mathrm{im}\ D^2$. Let $U := \ker D^2$ and $V := \mathrm{im}\ D^2$, then it is obvious that $D$ maps $U$ to $U$ and $V$ to $V$. Also, we have
$$D^2D_x = (D_xD_y + D_yD_x)D_x = D_x(D_xD_y + D_yD_x) = D_xD^2,$$
therefore $D_x$ also maps $U$ to $U$ and maps $V$ to $V$. We write $D'$ and $D_x'$ as restrictions of $D$ and $D_x$ to $U$, and similarly write $D''$ and $D_x''$ as restrictions of $D$ and $D_x$ to $V$.\\

Note that $H_D(M)$ is a quotient of $\ker D \subset \ker D^2 = U$, so we focus on our study to $U$. Since $(D')^2 = 0$, so $\mathrm{im}\ D' \subset \ker D'$ and $H_D(M) = \ker D'/\mathrm{im}\ D'$. Also, $(D_x')^2 = 0$ implies $\mathrm{im}\ D_x' \subset \ker D_x'$. Now for any irreducible $\widetilde{W}$-module $\nu$:
$$[\nu : \ker D']_{\widetilde{W}} + [\nu : \mathrm{im}\ D']_{\widetilde{W}} = [\nu : U]_{\widetilde{W}} = [\nu : \ker D_x']_{\widetilde{W}} + [\nu : \mathrm{im}\ D_x']_{\widetilde{W}}.$$
(Note that $[\nu : U]$ is finite, since $U$ is finite dimensional by Lemma 3.13 of \cite{C1}) \\

By Lemma \ref{thm:filter}, $[\nu : \ker D']_{\widetilde{W}} \leq [\nu : \ker D_x']_{\widetilde{W}}$ for all $\nu$. Hence, $[\nu : \mathrm{im}\ D']_{\widetilde{W}} \geq [\nu : \mathrm{im}\ D_x']_{\widetilde{W}}$ and consequently we have an inclusion of $W$-modules:
$$H_D(M) = \ker D'/\mathrm{im}\ D' \hookrightarrow \ker D_x'/\mathrm{im}\ D_x'$$
\ \ \ Finally, note that $\ker D_x/\mathrm{im}\ D_x = \ker D_x'/\mathrm{im}\ D_x' \oplus \ker D_x''/\mathrm{im}\ D_x''$, so $\ker D_x'/\mathrm{im}\ D_x'$ is naturally a subspace of the following space
$$\ker D_x/\mathrm{im}\ D_x \cong H^{\bullet}(\mf{h}^*, M) \otimes \chi.$$
Thus, the theorem is proved. \qed \\

\begin{remark}
In fact, the hypothesis of Theorem \ref{thm:incprop} is satisfied when $\Omega_{{\bf H}_{t,c}}$ (Definition \ref{eq:euler}) acts semisimply on $M$. This can be seen from the formula of $D^2$ given by Equation (4.24) of \cite{C1}. For the rest of the manuscript, we will apply Theorem \ref{thm:incprop} in this setting.
\end{remark}


We have an alternative proof of Theorem 3.14 in \cite{C1} for ${\bf H}$-modules $M$ which $\Omega_{\bf H}$ acts semisimply upon:
\begin{corollary}
Let ${\bf H} = {\bf H}_{t,c}$ with $t \neq 0$, and $M$ be an $\Omega_{\bf H}$-semisimple module. Suppose that $\widetilde{\sigma} \otimes \beta$ is an isotypical component in $H_D(M)$, where $\widetilde{\sigma}$ is an irreducible $\widetilde{W}$-module and $\beta$ is a $\mc{B}$-character.
Then the morphism $\zeta_d^*:\mathrm{Irr}(\widetilde{W}) \to \mathrm{Spec}\ \mc{B}$ satisfies the following condition
$$\zeta_d^*(\widetilde{\sigma}) = \beta.$$
\end{corollary}
\begin{proof}
By  Equation (4.12) of \cite{C1}, we have
\begin{align} \label{id:grad}
[\Omega_{\bf H}, x] = 2tx, \ \ [\Omega_{\bf H}, y] = -2ty.
\end{align}
Then it is easy to see our definition of $\mathrm{\mathbf{A}}$ in Section \ref{sec:voganconjecture} is the same as that in Equations (3.4)-(3.5) of \textit{loc. cit.} when ${\bf H}_{t,c}$ with $t \neq 0$. Suppose $\widetilde{\sigma} \otimes \beta \in H_D(M)$, then by the above Remark, we can apply Theorem \ref{thm:incprop} so that $\widetilde{\sigma} \otimes \beta \in \ker D_x/\mathrm{im} D_x$. Hence the result follows from Theorem \ref{cor:voganmorphism}.
\end{proof}

\section{Hodge decomposition for unitarizable modules}
In this section, we show that  Dirac cohomology is
isomorphic to Lie algebra cohomology up to a twist of a character for unitarizable modules. This follows from a Hodge
decomposition theorem for Dirac operators.

We set ${\bf H} := {\bf H}_{t,c}$ with $t, c(s) \in \bb{R}$ for all $s \in \mc{R}$.
Recall that a $\ast$-action is defined on ${\bf H}$ in \cite{ES}. For a suitable choice of the basis $x_i$, $y_i$ in $\mf{h}^*$ and $\mf{h}$ respectively, $\ast$ has the property that $x_i^* = y_i$, $y_i^* = x_i$.
A ${\bf H}$-module $M$ has a Hermitian $\ast$-invariant form if
$$(z\cdot m_1, m_2)_M = (m_1, z^* \cdot m_2)_M$$
for all $z \in {\bf H}$ and $m_1, m_2 \in M$. Furthermore, if the form is positive definite, we call $M$ a {\bf $\ast$-unitary module}.\\

The main theorem of this section is the following:
\begin{theorem}\label{thm:unitary}
Suppose the $M$ is a $\ast$-unitary ${\bf H}$-module. Then the injection in Theorem \ref{thm:incprop} is an isomorphism of $\widetilde{W}$-modules, i.e.
$$H_D(M) \cong H^{\bullet}(\mf{h}^*, M) \otimes \chi.$$
The same result holds for $H_{\bullet}(\mf{h}, M)$.
\end{theorem}

In order to prove Theorem \ref{thm:unitary}, we  need a Hermitian form on $M \otimes S$. In fact,
we can endow $S$ with a positive definite Hermitian form $(\cdot,\cdot)_S$ by the following: Let $\{ y_I := y_{i_1} \wedge \dots \wedge y_{i_k}\ |\  I = \{i_1, \dots, i_k\} \subset \{1,\dots,n \} \}$ be a basis of $\wedge^{\bullet} \mf{h} \cong S$. Then the Hermitian form on $S$ is defined by $(y_I, y_J) := \delta_{I,J}$. One can verify that $(x_i \cdot s_1, s_2) = -(s_1, y_i \cdot s_2)$ for all $s_1, s_2 \in S$ and all $i = 1, \dots, n$.

Now one can define Hermitian form $(\cdot,\cdot)_{M \otimes S}$ on $M \otimes S$ by:
$$(m_1 \otimes s_1 ,m_2 \otimes s_2)_{M \otimes S} = (m_1,m_2)_M(s_1,s_2)_S.$$

\begin{lemma} \label{thm:hodgelemma}
The adjoint of the half-Dirac operator $D_x$ is $D_x^* = - D_y$.
\end{lemma}
\begin{proof}
\begin{align*}
(D_x(m_1 \otimes s_1) ,m_2 \otimes s_2)_{M \otimes S} &= \sum_i (x_i \cdot m_1 \otimes y_i \cdot s_1 ,m_2 \otimes s_2)_{M \otimes S}\\
&= \sum_i (x_i \cdot m_1, m_2)_M (y_i \cdot s_1 , s_2)_S\\
&= - \sum_i (m_1, y_i \cdot m_2)_M (s_1 , x_i \cdot s_2)_S\\
&= (m_1 \otimes s_1 , -\sum_i (y_i \otimes x_i) \cdot (m_2 \otimes s_2))_{M \otimes S}.
\end{align*}
Thus, $D_x^*  = -D_y$.
\end{proof}
Similarly, $D_y^* = -D_x$ and we have
$$D^* = (D_x + D_y)^* = -D_y - D_x = -D.$$
\begin{proposition} \label{thm:hodgeprop}
Let $M$ be an irreducible ${\bf H}$-module satisfying the hypothesis given in Theorem \ref{thm:unitary}. Then we have
$$\ker D = \ker D^2 = \ker D_x \cap \ker D_y.$$
\end{proposition}
\begin{proof} It is easy to see $\ker D \subset \ker D^2$. Suppose $\widetilde{m} \in \ker D^2$. Then
$$0 = (D^2 \widetilde{m}, \widetilde{m}) = (D\widetilde{m}, D^*\widetilde{m}) = -(D\widetilde{m}, D\widetilde{m}).$$
It follows that $D\widetilde{m} = 0$ and hence $\widetilde{m} \in \ker D$. Also, since $\widetilde{m} \in \ker D = \ker D^2$,
$$0 = D^2\widetilde{m} = D_xD_y\widetilde{m} + D_yD_x \widetilde{m}$$
and $D_xD_y \widetilde{m} = - D_yD_x \widetilde{m}$. Applying $D_x$ to both sides, we have $D_x D_y D_x \widetilde{m} = 0$ and
$$0 = (D_x D_y D_x \widetilde{m} , D_x\widetilde{m}) = -(D_y D_x \widetilde{m}, D_yD_x \widetilde{m}).$$
It follows that $D_yD_x\widetilde{m} = 0$. Similarly, $0 = (D_y D_x \widetilde{m} , \widetilde{m}) = -(D_x\widetilde{m}, D_x\widetilde{m})$ and hence $D_x\widetilde{m} = 0$, i.e. $\widetilde{m} \in \ker D_x$. A similar argument also gives $\widetilde{m} \in \ker D_y$. Therefore, we conclude that
 $\ker D^2 \subset \ker D_x \cap \ker D_y$. For the other inclusion, we note that $D_x\widetilde{m} = D_y \widetilde{m} = 0$ implies $D^2\widetilde{m} = D_x D_y \widetilde{m} +  D_yD_x \widetilde{m} = 0 + 0 = 0$.
\end{proof}

We now prove a Hodge decomposition theorem for $M \otimes S$. First of all, we notice the following two facts:\\
\noindent Fact (i) -- We have $\mathrm{im}\ D_x \cap \mathrm{im}\ D_y = \{0\}$. If $\widetilde{m} \in \mathrm{im} D_x \cap \mathrm{im} D_y$, i.e. $\widetilde{m} = D_x \widetilde{n} = D_y \widetilde{n}'$, then $D_x \widetilde{m} = D_x^2 \widetilde{n} = 0 = D_x D_y \widetilde{n}'$. By the same argument as in Proposition \ref{thm:hodgeprop}, $D_x D_y \widetilde{n}' = 0$ implies $D_y \widetilde{n}' = \widetilde{m} = 0$.\\
\noindent  Fact (ii) -- We have $\ker D_x \perp \mathrm{im}\ D_y$. If $\widetilde{m} \in \ker D_x$, then $0 = (D_x \widetilde{m}, \widetilde{n}) = -(\widetilde{m},D_y \widetilde{n})$. Similarly, we have $\ker D_y \perp \mathrm{im}\ D_x$.
\begin{theorem}[Hodge decomposition] \label{thm:hodgethm}
Let $M$ be a $\ast$-unitary ${\bf H}$-module.  Then we have the following:\\
(a)\ $M \otimes S = \ker D \oplus \mathrm{im}\ D_x \oplus \mathrm{im}\ D_y$.\\
(b)\ $\ker D_x = \ker D \oplus \mathrm{im}\ D_x$, $\ker D_y = \ker D \oplus \mathrm{im}\ D_y$.
\end{theorem}
\begin{proof}
(a) By Proposition \ref{thm:hodgeprop}, $\ker D = \ker D^2$ and hence we only need to prove $\mathrm{im}\ D^2 = \mathrm{im}\ D_x \oplus \mathrm{im}\ D_y$ (Recall $\mathrm{im}\ D_x \cap \mathrm{im}\ D_y = \{0\}$ by Fact (i)). \\
One inclusion is simple -- suppose $\widetilde{m} = D_xD_y \widetilde{n} + D_yD_x \widetilde{n} \in \mathrm{im}\ D^2$, then it is automatically in $\mathrm{im}\ D_x \oplus \mathrm{im}\ D_y$.\\
Now suppose $\widetilde{m} \in \mathrm{im}\ D_x$, then $\widetilde{m} \perp \ker D_y$. But $\ker D^2 \subset \ker D_y$ by Proposition \ref{thm:hodgeprop}, hence $\widetilde{m} \in (\ker D^2)^{\perp}$. Since $M \otimes S = \ker D^2 \oplus \mathrm{im}\ D^2$, and our Hermitian product is positive definite, $(\ker D^2)^{\perp} = \mathrm{im}\ D^2$ and consequently $\widetilde{m} \in \mathrm{im}\ D^2$. Similarly, one can prove $\mathrm{im}\ D_y \subset \mathrm{im}\ D^2$ and therefore (a) is proved.\\

\noindent (b) By Fact (ii) above, $\ker D_x \subset (\mathrm{im}\ D_y)^{\perp}$. Now (a) says $(\mathrm{im}\ D_y)^{\perp} = \ker D \oplus \mathrm{im}\ D_x$. So we have the inclusion $\ker D_x \subset \ker D \oplus \mathrm{im}\ D_x$.\\
For the other inclusion, note that by Proposition \ref{thm:hodgeprop} we have $\ker D \subset \ker D_x$, and $\mathrm{im}\ D_x \subset \ker D_x$ since $D_x^2 = 0$ so the inclusion must be an equality.\\
The second part of the statement is analogous to the first part, and we omit the proof of it.
\end{proof}

\noindent \textit{Proof of Theorem \ref{thm:unitary}.} Since $\ker D = \ker D^2$, $\ker D \cap \mathrm{im}\ D = \{0\}$ and hence $H_D(M) = \ker D$. By the first equality of Theorem \ref{thm:hodgethm}(b),
$$H_D(M) = \ker D \cong \ker D_x/\mathrm{im}\ D_x \cong H^{\bullet}(\mf{h}^*, M) \otimes \chi.$$
Similarly,
$$H_D(M) = \ker D \cong \ker D_y/\mathrm{im}\ D_y \cong H_{\bullet}(\mf{h}, M) \otimes \chi$$
by the second equality of Theorem \ref{thm:hodgethm}(b). \qed

\section{Lie algebra cohomology for ${\bf H}_{t,c}$ with $t=1$}
We note that the mapping $x\mapsto \lambda x, y\mapsto \lambda y, w\mapsto w$ for $x\in\mf{h}^*$, $y\in \mf{h}$
and $w\in W$ induces an algebra isomorphism
$$ {\bf H}_{t,c}\cong {\bf H}_{\lambda^2 t,\lambda^2 c}.$$
Therefore,  we need to consider only two cases $t=1$ and $t=0$ up to equivalence.

In this section, we assume $t=1$ and set ${\bf H} := {\bf H}_{1,c}$. Recall from \cite{GGOR}, the category $\mc{O}$ for ${\bf H} := {\bf H}_{1,c}$ is defined. For any irreducible $W$-module $\sigma$, the standard module $M(\sigma)$ is defined to be
$$M(\sigma) = {\bf H} \otimes_{S(\mf{h}) \rtimes \bb{C}[W]} \sigma.$$
As $S(\mf{h}^*) \rtimes \bb{C}[W]$-module, it is isomorphic to $S(\mf{h}^*) \otimes \sigma$, and it has a unique irreducible quotient $L(\sigma)$. For most values of $c$, the standard module is irreducible, i.e. $M(\sigma) = L(\sigma)$.\\

The standard module $M(\sigma) \cong S(\mf{h}^*) \otimes \sigma$ has an $\Omega_{{\bf H}}$-eigenspace decomposition. More precisely, $\Omega_{{\bf H}}$ acts semisimply on $M(\sigma)$
with lowest weight vectors being of the form $1 \otimes v_{\sigma} \in M(\sigma)$ and $\Omega_{{\bf H}}$ acting by the scalar $a_0$. By Equation \eqref{id:grad}, if $m \in M$ is an eigenvector of $\Omega_{{\bf H}}$ of eigenvalue $r$, then $x \cdot m$ is also an eigenvector of eigenvalue $r+2$ for any $x \in \mf{h}^*$. Therefore, by letting $a_k = a_0 + 2k$, we have an $\Omega_{{\bf H}}$-eigenspace decomposition
$$M(\sigma) = \bigoplus M_{a_i},\ \ \ \text{with }\ M_{a_i} = S^k(\mf{h}^*) \otimes \sigma.$$
since every submodule of $J \subset M(\sigma)$ is graded, i.e. $J = \bigoplus J_{a_i}$ with $J_{a_i} \subset M_{a_i}$, we conclude that $\Omega_{{\bf H}}$ acts semisimply on every subquotient $L$ of $M(\sigma)$ and Theorem \ref{thm:incprop} applies.\\

\begin{proposition}\label{thm:incthm}
Suppose $M \in \mc{O}$ has a {\bf BGG resolution}, i.e.,
there is a exact sequence of ${\bf H}$-modules
\begin{align}\label{eq:BGG}
0 \to \bigoplus_{j=1}^{j_n} M(\sigma_{n,j}) \to \bigoplus_{j=1}^{j_{n-1}} M(\sigma_{n-1,j}) \to \dots \to \bigoplus_{j=1}^{j_0} M(\sigma_{0,j}) \to M \to 0.
\end{align}
Then we have inclusions of $W$-modules for all $i \geq 0$:
$$H_{i}(\mf{h}^*,M) \leq \bigoplus_{j = 1}^{j_i} \sigma_{i,j}.$$
Moreover, if $\sigma_{i,j} \ncong \sigma_{i+1,j'}$ for all $i,j,j'$, then the above inclusions are isomorphisms.
\end{proposition}
\begin{proof}
We restrict our attention to the $R = S(\mf{h}^*)$-module structure of a ${\bf H}$-module $M$. By \eqref{eq:BGG}, there is a free $R$-resolution of $M$ by:
$$0 \to R \otimes (\bigoplus_{j=1}^{j_n} \sigma_{n,j}) \to R \otimes (\bigoplus_{j=1}^{j_{n-1}} \sigma_{n-1,j}) \to \dots \to R \otimes (\bigoplus_{j=1}^{j_0} \sigma_{0,j}) \to M \to 0.$$
Therefore, $Tor_{\bullet}^R(\bb{C}, M)$ can be computed by tensoring $\bb{C} \otimes_{R} \bullet$ to the above sequence, which gives a complex of $W$-modules
\begin{align}\label{wcomplex}
0 \to \bigoplus_{j=1}^{j_n} \sigma_{n,j} \to \bigoplus_{j=1}^{j_{n-1}} \sigma_{n-1,j} \to \dots \to \bigoplus_{j=1}^{j_0} \sigma_{0,j} \to 0.
\end{align}
Hence, $Tor_i^R(\bb{C},M) \leq \bigoplus_{j=1}^{j_i} \sigma_{i,j}$ and by usual homological algebra argument, we have $Tor_i^R(\bb{C},M) \cong H_i(\mf{h}^*,M)$ and the first part of the proposition is proved.\\
Assume now that $\sigma_{i,j} \neq \sigma_{i+1,j'}$ for all $i,j,j'$. Then the differentials in the complex \eqref{wcomplex} must be all zeros. So in this case
$$H_i(\mf{h}^*,M) \cong Tor_i^R(\bb{C},M) \cong \bigoplus_{j=1}^{j_i} \sigma_{i,j}.$$
Therefore, the proposition is proved.
\end{proof}

It follows from the Poincar\'{e} duality (Proposition \ref{poincare}) for $\mf{c} = \mf{h}^*$ that
\begin{align} \label{eq:poincoho}
H^{\bullet}(\mf{h}^*,M) \otimes \chi \cong H_{\bullet}(\mf{h}^*,M) \otimes \chi^{-1} \leq \bigoplus_{i = 0}^n \bigoplus_{j = 1}^{j_i} \sigma_{i,j} \otimes \chi^{-1},
\end{align}
and the above inclusion is an isomorphism if $\sigma_{i,j} \ncong \sigma_{i+1,j'}$ for all $i,j,j'$.

The following corollary gives precisely the Dirac cohomology of a standard module $M(\sigma)$. This refines Proposition 5.6 of \cite{C1}:
\begin{corollary}\label{thm:std}
The standard module $M(\sigma)$ has Dirac cohomology
$$H_D(M(\sigma)) \cong \sigma \otimes \chi^{-1}.$$
\end{corollary}
\begin{proof}
The BGG resolution of $M = M(\sigma)$ is obviously
$$0 \to M(\sigma) \to M \to 0.$$
By Proposition \ref{thm:incthm}, we have
$$H_i(\mf{h}^*, M(\sigma)) = 0\ \ \ \text{for }\ i > 0,\ \ \ \ \ H_0(\mf{h}^*, M(\sigma)) = \sigma.$$
Since $M(\sigma)$ is $\Omega_{\mathbf{H}}$-semisimple, we can apply Theorem \ref{thm:incprop} and Equation (\ref{eq:poincoho}) to get $H_D(M(\sigma))$ $\leq$ $H_{\bullet}(\mf{h}^*, M(\sigma)) \otimes \chi^{-1}$ $=$ $\sigma \otimes \chi^{-1}$. On the other hand, Proposition 5.6 of \cite{C1} says $\sigma \otimes \chi^{-1} \in H_D(M(\sigma))$. Therefore the result follows.
\end{proof}

As another application of Proposition \ref{thm:incthm}, we consider the finite-dimensional ${\bf H}$-modules $L(\mathrm{triv})$ studied in \cite{BEG}. In fact, Proposition 1.6 of \cite{BEG} gives the BGG resolution of $L(\mathrm{triv})$ by
$$0 \to M(\wedge^n \mf{h}^*) \to M(\wedge^{n-1} \mf{h}^*) \to \dots \to M(\wedge^1 \mf{h}^*) \to M(\mathrm{triv}) \to L(\mathrm{triv}) \to 0.$$
For any Weyl group $W$, all $\wedge^i \mf{h}^* \cong \wedge^i \mf{h}$ are distinct irreducible modules, so Proposition \ref{thm:incthm} applies and
$$H_{i}(\mf{h}^*,L(\mathrm{triv})) \cong \wedge^{i}\mf{h}.$$
One would like to relate the above result with Dirac cohomology $H_D(L(\mathrm{triv}))$. Indeed, by Theorem \ref{thm:incprop} and Equation (\ref{eq:poincoho}), we have
$$H_D(L(\mathrm{triv})) \leq H^{\bullet}(\mf{h}^*,L(\mathrm{triv})) \otimes \chi = H_{\bullet}(\mf{h}^*,L(\mathrm{triv})) \otimes \chi^{-1} = \wedge^{\bullet}\mf{h} \otimes \chi^{-1}.$$
To see the inclusion is an isomorphism, we need the following:

\begin{theorem}\label{thm:parity}
Let $M$ be an $\Omega_{{\bf H}}$-semisimple ${\bf H}$-module with an eigenspace decomposition $M = \bigoplus_{\lambda} M_{\lambda}$. If $M$ satisfies the {\bf parity condition}, i.e.
$$[H^{even}(\mf{h}^*,M) : H^{odd}(\mf{h}^*,M)]_W = 0,$$
then $H_D(M) \cong H^{\bullet}(\mf{h}^*,M) \otimes \chi$. The analogous statement holds also for $H_{\bullet}(\mf{h},M)$.
\end{theorem}

Applying Theorem \ref{thm:parity} to the finite-dimensional modules $L(\mathrm{triv})$, we get the following:
\begin{corollary}
Let $M = L(\mathrm{triv})$ be a finite-dimensional ${\bf H}$-module in \cite{BEG}, then
$$H_D(L(\mathrm{triv})) \cong \wedge^{\bullet}\mf{h} \otimes \chi^{-1}$$
\end{corollary}
\begin{proof}
Since all $H^i(\mf{h}^*,M) = H_{n-i}(\mf{h}^*, M) \otimes {\det}_{\mf{h}} = \wedge^{n-i} \mf{h} \otimes \wedge^n \mf{h} \cong \wedge^i \mf{h}$ are distinct $W$-modules, the hypothesis of Theorem \ref{thm:parity} is satisfied. Hence the result follows.
\end{proof}

As before, we only prove Theorem \ref{thm:parity} for $H^{\bullet}(\mf{h}^*,M)$. The proof is analogous to that of \cite{HX}. However, the proof given there requires the fact that weight spaces of $M \otimes S$ are finite-dimensional, yet in our setting the multiplicities of irreducible $W$-modules can be infinite. So we have to decompose $M \otimes S$ into a direct sum of finite-dimensional $W$-modules, where $D_x, D_y$ and therefore $D$ will act on each summand individually.\\

Identify $S$ with $\wedge^{\bullet}\mf{h}$ as vector spaces. Under the hypothesis, we have the $\Omega_{{\bf H}}$-eigenspace decomposition of $M = \bigoplus_a M_a$. Define a bi-grading on $M \otimes S = \bigoplus_{a,j} (M \otimes S)_{a,j}$ by
\begin{align}\label{def:bigrade}
(M \otimes S)_{a,l} = M_a \otimes \wedge^l \mf{h}.
\end{align}
By \eqref{id:grad} and the definitions of $D_x$ and $D_y$, $D_x$ maps $(M \otimes S)_{a,l}$ to $(M \otimes S)_{a+2,l+1}$, and $D_y$ maps $(M \otimes S)_{a,l}$ to $(M \otimes S)_{a-2,l-1}$.

\begin{definition}
Given the bi-grading of $M \otimes S$ in \eqref{def:bigrade}, set
$$U_r := \bigoplus_{a - 2l = r} M_{a,l},\ \ \ U_r^+ := \bigoplus_{a-4k = r} M_{a,2k},\ \ \ U_r^- := \bigoplus_{a-4k+2 = r} M_{a,2k-1}.$$
We can express $U_r$ as a complex :
$$0 \to M_{r,0} \xrightarrow{D_x} M_{r+2,1} \xrightarrow{D_x} \dots \xrightarrow{D_x} M_{r+2n,n} \to 0$$
We define
$$H^i(U_r,D_x) := {\ker(D_x: M_{r+2i,i} \to M_{r+2i+2,i+1})}/{\mathrm{im}(D_x: M_{r+2i-2,i-1} \to M_{r+2i,i})}.$$ Then we have
$$\ker D_x/\mathrm{im}\ D_x = \bigoplus_r H^{\bullet}(U_r, D_x) = \bigoplus_r  (H^{even}(U_r, D_x)  \oplus H^{odd}(U_r, D_x) ).$$
Replacing $D_x$ with $D_y$, we define $H_i(U_r, D_y)$ analogously.\\
As for $D$, let
$$D_r : U_r \to U_r;\ \ D_r^{\pm}: U_r^{\pm} \to U_r^{\mp}$$
be the restriction of $D$ on $U_r$, and
$$H_{D}(U_r) := \ker D_r/(\ker D_r \cap \mathrm{im}\ D_r),\ \ H_{D}^{\pm}(U_r) := \ker D_r^{\pm}/(\ker D_r^{\pm} \cap \mathrm{im}\ D_r^{\mp}).$$
Then
$$H_D(M) = \bigoplus_r H_{D}(U_r) = \bigoplus_r (H_{D}^+(U_r) \oplus H_{D}^-(U_r)).$$
\end{definition}
So we can work on the finite dimensional space $U_r$ instead of the whole space $M \otimes S$.

\begin{lemma} \label{thm:urinc}
In the above setting, we have the following\\
(a)  inclusions of $\widetilde{W}$-modules
$$H_{D}^{+}(U_r) \hookrightarrow H^{even}(U_r,D_x),\ \ H_{D}^{-}(U_r) \hookrightarrow H^{odd}(U_r,D_x);\ \ \  \text{and}$$
(b) identity for virtual $\widetilde{W}$-modules,
\begin{align*}
H_{D}^+(U_r) - H_{D}^-(U_r) = H^{even}(U_r,D_x) - H^{odd}(U_r,D_x).
\end{align*}
\end{lemma}

\begin{proof}
(a) Follows directly from Theorem \ref{thm:incprop} when we restrict to $U_r$.\\
(b) By hypothesis and Theorem \ref{thm:dsquare}(e), $D_r^+ \circ D_r^-$ and $D_r^- \circ D_r^+$ are both semisimple. Then using linear algebra, e.g. Proposition 5.2 of \cite{HX}, one can conclude that $H_{D}^+(U_r) - H_{D}^-(U_r)$ is equal to $U_r^+ - U_r^-$ as $W$-modules. However, $U_r^+ - U_r^-$ is the Euler characteristic of the half-Dirac complex in $U_r$, hence it is also equal to $H^{even}(U_r,D_x) - H^{odd}(U_r,D_x)$.
\end{proof}

\noindent \textit{Proof of Theorem \ref{thm:parity}}. Recall that
\begin{align*}
\ker D_x/\mathrm{im}\ D_x & = \bigoplus_r  (H^{even}(U_r, D_x)  \oplus H^{odd}(U_r, D_x) ) \\
&\cong (H^{even}(\mf{h}^*,M) \otimes \chi) \oplus (H^{odd}(\mf{h}^*,M) \otimes \chi).
\end{align*}
By hypothesis, the two summands on the right hand side are disjoint. Also,
$$\bigoplus_r H^{even}(U_r,D_x) = H^{even}(\mf{h}^*,M) \otimes \chi,\ \ \ \ \ \bigoplus_r H^{odd}(U_r,D_x) = H^{odd}(\mf{h}^*,M) \otimes \chi$$
implies $H^{even}(U_r,D_x)$ and $H^{odd}(U_r,D_x)$ are disjoint as well. By Lemma \ref{thm:urinc}(b),
we have
$$H_D^+(U_r) \supset H^{even}(U_r,D_x) \text{ and }H_D^-(U_r) \supset H^{odd}(U_r,D_x).$$
It follows from  Lemma \ref{thm:urinc}(a), the inclusions must be equalities and we have
$$H_D^+(U_r) = H^{even}(U_r,D_x), \ \ \ H_D^-(U_r) = H^{odd}(U_r,D_x).$$
Hence, $H_D(U_r) = H^{\bullet}(U_r,D_x)$. Adding up the $r$'s, we have $H_D(M) = \ker D_x/\mathrm{im}\ D_x$ and the result follows. \qed

To conclude this section, we make the following:
\begin{conjecture}\label{conj:parity}
For any irreducible modules $L(\sigma)$, we have the $\widetilde{W}$-module isomorphisms
$$H_D(L(\sigma)) \cong H^{\bullet}(\mf{h}^*,L(\sigma)) \otimes \chi \cong H_{\bullet}(\mf{h},L(\sigma))\otimes \chi .$$
\end{conjecture}

\section{Lie algebra cohomology for ${\bf H}_{t,c}$ with $t=0$}
In this section, we set ${\bf H} := {\bf H}_{0,c}$. We study the center $Z({\bf H})$ of ${\bf H}$ first of all. As opposed to the $t = 1$ case whose center only contains the constants, $\bf H$ has a large center as shown in the following lemma:
\begin{lemma}\label{lem:gamma}
The center $Z({\bf H})$ of ${\bf H}$ is a free $S(\mf{h})^W \otimes S(\mf{h}^*)^W$-module of rank $|W|$.
More precisely, there exists $\gamma_i \in {\bf H}^{\bb{C}^*}$ (Recall the $\bb{C}^*$-action on ${\bf H}$ in Definition \ref{def:cstar}) with $\gamma_1 = 1$ such that
$$Z({\bf H}) \cong \bigoplus_{i = 1}^{|W|} (S(\mf{h})^W \otimes S(\mf{h}^*)^W)\cdot \gamma_i.$$
\end{lemma}
\begin{proof}
The first statement of the lemma is in Proposition 4.15 of \cite{EG}. For the last statement, note that by the commutation relations of ${\bf H}$, $a \in {\bf H}^{\bb{C}^*} \Leftrightarrow gr(a) \in gr({\bf H})^{\bb{C}^*}$, where the grading is the one given in Section \ref{sec:voganconjecture}. So we only need to prove $gr(\gamma_i) \in gr({\bf H})^{\bb{C}^*}$ for all $i$. In this setting, the proof of Proposition 4.15 \cite{EG} gives the structure of $gr(\gamma_i)$ precisely, with all $gr(\gamma_i)$ can be chosen to be in $gr({\bf H})^{\bb{C}^*}$. Hence, the Lemma follows.
\end{proof}
Let $\mf{m}_+$ be the maximal ideal of $S(\mf{h})^W \otimes S(\mf{h}^*)^W$ consisting of all non-constant polynomials. The \textbf{restricted rational Cherednik algebra} $\overline{{\bf H}}$ is defined to be the quotient ${\bf H}/\mf{m}_+{\bf H}$. Section 6 of \cite{G} says that $\overline{{\bf H}}$ is of rank $|W|^3$, and the center of $\overline{{\bf H}}$ is $Z({\bf \overline{H}}) = Z({\bf H})/\mf{m}_+Z({\bf H})$. Therefore, Lemma \ref{lem:gamma} says $\dim Z({\bf \overline{H}}) = |W|$.\\

Following to \cite{G2}, a family of
$\bf H$-modules called {\bf baby Verma modules} is defined as follows.
For each $\sigma \in\mathrm{Irr}(W)$, set
$$\overline{M}(\sigma) := {\bf H} \otimes_{(S(\mf{h})^W\otimes S(\mf{h^*})) \rtimes \bb{C}[W]} \sigma,$$
where $S(\mf{h})^W\otimes S(\mf{h^*})$ acts by evaluating at $(0,0)$.
Since $\mf{m}_+$ annihilates $\overline{M}(\sigma)$, so $\overline{M}(\sigma)$
becomes an $\overline{{\bf H}}$-module. By the results in \cite{G2}, it has a unique irreducible head $\overline{L}(\sigma)$.
Clearly,  the dimension of $\overline{L}(\sigma)$ is bounded above by
dimension of $\overline{M}(\sigma)$ which is equal to $|W|\cdot \dim \sigma$. Also, $\overline{L}(\sigma)\cong \overline{L}(\lambda)$ if and only if
$\sigma\cong \lambda$, and
$\{\overline{L}(\sigma)|\sigma\in\mathrm{Irr}(W)\} $ is a complete set of irreducible
$\overline{{\bf H}}$-modules up to equivalence. \\

The central character of all such $\overline{L}(\sigma)$ defines a map
$$\Theta:\mathrm{Irr}(W) \to \text{Spec}\ Z(\overline{\bf H}).$$
The elements in the pre-image of $\Theta$ of an element in $\text{Spec}\ Z(\overline{\bf H})$ form a \textbf{Calegero-Moser cell}.
This defines a partition of $W$. The Calogero-Moser cell carries information on the smoothness of the variety
$X_c := \text{Spec}\ Z({\bf H})$. More precisely, if $X_c$ is smooth, then all $\overline{L}(\sigma)$ has dimension $|W|\cdot \dim \sigma$, and every Calogero-Moser cell is a singleton.\\

By Remark 4.10 of \cite{C1}, all $\overline{L}(\sigma)$ are $\Omega_{{\bf H}}$-semisimple, and Theorem \ref{thm:incprop} of inclusion of Dirac cohomology into Lie algebra cohomology can be applied to $M$ (regarded as an $\bf H$-module). We set
$$\mc{B} := Z({\bf H})^{\bb{C}^*} \otimes 1 = (\bigoplus_{i = 1}^{|W|} \bigoplus_{k \geq 0} (S^k(\mf{h})^W \otimes S^k(\mf{h}^*)^W)\cdot \gamma_i) \otimes 1.$$
Then $\mc{B}$ is in $(Z({\bf H}) \otimes 1) \cap \mathrm{\mathbf{A}}$, so $\mc{B} \subset \ker \delta_d \subset \mathrm{\mathbf{A}}^{W}$ (or $\mc{B} \subset \ker \delta_{\partial} \subset \mathrm{\mathbf{A}}^{W}$) satisfies the hypothesis of Theorem \ref{cor:wonly}. So we can define the homomorphism
$$\zeta_d: Z({\bf H})^{\bb{C}^*} \otimes 1 \to \Delta(\bb{C}[\widetilde{W}]^{\widetilde{W}}).$$

\begin{proposition}
We have
$$(\mf{m}_+Z({\bf H}))^{\bb{C}^*} \otimes 1 \in \mathrm{im}\ \delta_d,$$
hence $\zeta_d$ descends to the homomorphism:
$$\zeta_d: Z({\bf H})^{\bb{C}^*} \otimes 1/(\mf{m}_+Z({\bf H}))^{\bb{C}^*} \otimes 1 \cong Z(\overline{\bf H}) \to \Delta(\bb{C}[\widetilde{W}]^{\widetilde{W}}).$$
\end{proposition}
\begin{proof}
By Lemma \ref{lem:gamma}, all elements in $(\mf{m}_+Z({\bf H}))^{\bb{C}^*}$ are of the form $\sum_i f_ig_i \gamma_i$, where $f_i \in S(\mf{h})^W_+$, $g_i \in S(\mf{h}^*)^W_+$ are of the same (positive) degree.
We first show that $g \otimes 1 \in \mathrm{im}\ \delta_d$ for any $g \in S(\mf{h}^*)^W_+$: Consider the map
$$\delta_d : (S(\mf{h}^*) \otimes \wedge^{\bullet}\mf{h}^*)^W \to (S(\mf{h}^*) \otimes \wedge^{\bullet}\mf{h}^*)^W$$
given by $\delta_d(a) = D_xa - \epsilon(a)D_x$. Since $g \in Z({\bf H})$, it is easy to see that $g \otimes 1 \in \ker \delta_d$. Also, from the proof of Proposition \ref{prop:bardeltadecomp}, we have seen that the map $\delta_d$ is an exact differential on non-zero degrees. Since $g$ is of positive degree, $g \otimes 1 \in \ker \delta_d = \mathrm{im}\ \delta_d$ as follows.\\
If $g_i \otimes 1 = D_xa - \epsilon(a)D_x$, then for $f_ig_i\gamma_i \otimes 1 \in (\mf{m}_+Z({\bf H}))^{\bb{C}^*}$, and
\begin{align*}
f_ig_i\gamma_i \otimes 1 &=  (f_i \otimes 1)(g_i \otimes 1)(\gamma_i \otimes 1)\\
&= (f_i \otimes 1)(D_x a - \epsilon(a) D_x)(\gamma_i \otimes 1)\\
&= D_x (f_i \otimes 1)a(\gamma_i \otimes 1) - \epsilon( (f_i \otimes 1)a (\gamma_i \otimes 1)) D_x,
\end{align*}
where the last equality comes from that fact that both $f_i \otimes 1, \gamma_i \otimes 1$ commute with $D_x$.
Hence,  $f_ig_i\gamma_i \otimes 1 \in \mathrm{im}\ \delta_d$ as required.
\end{proof}

The morphism $\zeta_d^*: \mathrm{Irr}(\widetilde{W}) \to \mathrm{Spec}Z({\bf \overline{H}})$
relates the central characters of an irreducible $\overline{\bf H}$-module and its Lie algebra cohomology.
\begin{theorem} \label{thm:babyvogan}
Let $M$ be an irreducible $\overline{\bf H}$-module with central character $\beta \in \mathrm{Spec}\ Z({\bf \overline{H}})$.
Suppose $\nu$ is an irreducible $W$-module appearing in $H^{\bullet}(\mf{h}^*, M)$. Then we have
$$\beta = \zeta_d^*(\nu \otimes \chi).$$
\end{theorem}
\begin{proof}
The proof is similar to Theorem \ref{cor:wonly}.
\end{proof}

We can now relate the maps $\Theta$ and $\zeta_d^*$. Combined with Theorem \ref{thm:incprop}, part (b) of the following Corollary gives an alternative proof of Corollary 5.10 of \cite{C1}.
\begin{corollary} \label{cor:calogero} \mbox{}
(a)\ For any $\sigma \in\mathrm{Irr}(W)$,
$$\Theta(\sigma) = \zeta_d^*(\sigma \otimes \chi^{-1}).$$
(b)\ If $H^{\bullet}(\mf{h}^*,\overline{L}(\sigma)) \cong \bigoplus_i \nu_i$ as $W$-modules with $\nu_i \in\mathrm{Irr}(W)$,
then all such $\nu_i \otimes {\det}_{\mf{h}^*}$'s belong to the same Calogero-Moser cell.
\end{corollary}
\begin{proof}
(a)\ We claim that the $W$-module $\sigma \otimes {\det}_{\mf{h}} \cong (1 \otimes \sigma) \otimes \wedge^{\dim \mf{h}} \mf{h}$ in $\overline{L}(\sigma) \otimes \wedge^{\bullet}\mf{h}$ is in $H^{\bullet}(\mf{h}^*,M)$. Indeed, since $\wedge^{\dim \mf{h}} \mf{h}$ is in its top degree, $D_x((1 \otimes \sigma) \otimes \wedge^{\dim \mf{h}} \mf{h}) = 0$. Also, noting that $(1 \otimes \sigma)$ has zero degree on its $\mf{h}$ and $\mf{h}^*$ factor, it cannot be in the image of $\mathrm{im} D_x$. So $(1 \otimes \sigma) \otimes \wedge^{\dim \mf{h}}$ must be non-zero in $\ker D_x/\mathrm{im}\ D_x = H^{\bullet}(\mf{h}^*,M)$.

Taking $M = \overline{L}(\sigma)$ and $\nu = \sigma \otimes {\det}_{\mf{h}}$ in Theorem \ref{thm:babyvogan}, we have
$$\Theta(\sigma) = \zeta_d^*(\sigma \otimes {\det}_{\mf{h}} \otimes \chi) = \zeta_d^*(\sigma \otimes \chi^{-1}).$$
(b)\ Suppose $\nu \in H^{\bullet}(\mf{h}^*,\overline{L}(\sigma))$. By Theorem \ref{thm:babyvogan},
$$\Theta(\sigma) = \zeta_d^*(\nu \otimes \chi).$$
On the other hand, (a) says $\zeta_d^*(\nu \otimes \chi) = \zeta_d^*((\nu \otimes {\det}_{\mf{h}^*}) \otimes \chi^{-1}) = \Theta(\nu \otimes {\det}_{\mf{h}^*})$. Hence $\nu \otimes {\det}_{\mf{h}^*}$, $\sigma $ are in the preimage of the same element in $\mathrm{Spec}\ Z({\bf \overline{H}})$, i.e. they are in the same Calogero-Moser cell.
\end{proof}
Since Theorem \ref{cor:wonly} holds for both $D_x$ and $D_y$, the above results also hold if we replace $H^{\bullet}(\mf{h}^*,M)$ with $H_{\bullet}(\mf{h},M)$. We skip the proofs here.\\

As in the case of $t=1$, we make the following conjecture:
\begin{conjecture} For irreducible $\overline{{\bf H}}$-module $\overline{L}(\sigma)$, there are $\widetilde{W}$-module isomorphisms
$$H_D(\overline{L}(\sigma)) \cong H^{\bullet}(\mf{h}^*,\overline{L}(\sigma)) \otimes \chi \cong H_{\bullet}(\mf{h},\overline{L}(\sigma)) \otimes \chi.$$
\end{conjecture}


\begin{thebibliography}{GGOR}

\bibitem[AS]{AS} M. Atiyah, W. Schmid, \emph{A geometric
construction of the discrete series for semisimple Lie groups}
Invent. Math. \textbf{42} (1977), 1--62, \textbf{54} (1979), 189--192.




\bibitem[BCT]{BCT}
D. Barbasch, D. Ciubotaru, P. Trapa,
\emph{Dirac cohomology for graded affine Hecke algebras},
Acta Math. \textbf{209} (2012), 197--227.



\bibitem[BEG]{BEG} Yu. Berest, P. Etingof, V. Ginzburg,
\emph{Finite-dimensional representations of rational Cherednik algebras},
Int. Math. Res. Not. \textbf{19} (2003), 1053--1088.


\bibitem[CO]{CO}
W. Casselman, M. Osborne,
\emph{The $\mf{n}$-cohomology of the representations with an infinitesimal characters},
Compositio Math., \textbf{31} (1975), 219--227.




\bibitem[C]{C1}
D. Ciubotaru,
\emph{Dirac cohomology for symplectic reflection groups},
Selecta Math. (N.S.) \textbf{22} (2016), no. 1, 111--144.



\bibitem[EG]{EG}
P. Etingof, V. Ginzburg,
\emph{Symplectic reflection algebras, Calogero-Moser space, and deformed Harish-Chandra homomorphism},
Invent. Math. \textbf{147} (2002), no. 2, 243--348.



\bibitem[ES]{ES}
P. Etingof, E. Stoica,
\emph{Unitary representations of rational Cherednik algebras},
Represent. Theory \textbf{13} (2009), 349--370.

\bibitem[G1]{G2}
I. Gordon,
\emph{Baby Verma modules for rational Cherednik algebras},
Bull. London Math. Soc. \textbf{35} (2003), no. 3, 321--336.


\bibitem[G2]{G} I. Gordon,
\emph{Symplectic reflection algebras}, in
\emph{Trends in Representation Theory of Algebras and Related Topics},
EMS Ser. Congr. Rep., Eur. Math. Soc., Zurich, 2008.





\bibitem[GGOR]{GGOR}
V. Ginzburg, N. Guay, E. Opdam, R. Rouquier,
\emph{On the category $\mc{O}$ for rational Cherednik algebras},
Invent. Math. \textbf{154} (2003), 617--651.

\bibitem[H]{H} J.-S. Huang,
\emph{Dirac cohomology, elliptic representations and endoscopy},  in book
\emph{Representations of Reductive Groups, in honor of the 60th Birthday of David Vogan},
edited by M. Newins and P. Trapa,  Birkh\"auser series Progress in Mathematics 312, Springer, Swizterland, 2015, 241-276.


\bibitem[HP1]{HP2}
J.-S. Huang, P. Pand\v{z}i\'{c},
\emph{Dirac cohomology, unitary representations and a proof of a conjecture of
Vogan}, J. Amer. Math. Soc. \textbf{15} (2002), 185--202.

\bibitem[HP2]{HP3}
J.-S. Huang, P. Pand\v{z}i\'{c},
\emph{Dirac cohomology for Lie Superalgebras}, Transform. Group \textbf{10} (2005), 201--209.

\bibitem[HP3]{HP1}
J.-S. Huang, P. Pand\v{z}i\'{c},
\emph{Dirac operators in Representation Theory},
Mathematics: Theory and Applications, Birkhauser,
2006.

\bibitem[HPR]{HPR}
J.-S. Huang, P. Pand\v{z}i\'{c}, D. Renard,
\emph{Dirac operarors and Lie algebra cohomology},
arXiv:math/0503582, 2005.

\bibitem[HPR']{HPR'}
J.-S. Huang, P. Pand\v{z}i\'{c}, D. Renard,
\emph{Dirac operators and Lie algebra cohomology}, Represent. Theory \textbf{10}
(2006), 299--313.

\bibitem[HPZ]{HPZ}
J.-S. Huang, P. Pand\v zi\'c, F.-H. Zhu,
\emph{Dirac cohomology, K-characters and branching laws}, Amer. J. Math. \textbf{135} (2013),  1253--1269.

\bibitem[HX]{HX}
J.-S. Huang, W. Xiao,
\emph{Dirac cohomology for highest weight modules},
Selecta Math. \textbf{18} (2012), 803--824.

\bibitem[KMP]{KMP} V. G. Kac, P. M\"oseneder Frajria and P. Papi, \emph{Multiplets of
representations, twisted Dirac operators and Vogan's conjecture
in affine setting}, Adv. Math. \textbf{217} (2008), 2485--2562.

\bibitem[K1]{K}
B. Kostant,
\emph{Lie algebra cohomology and the generalized Borel-Weil Theorem},
Ann. Math. \textbf{74} (1960), 329--387.

\bibitem[K2]{Ko1}
B. Kostant,
\emph{A generalization of the Bott-Borel-Weil Theorem and Euler number multiplets of Representations}, Lett. Math. Phys.
\textbf{52} (2000), 61--78.


\bibitem[K3]{Ko3}
B. Kostant,
\emph{Dirac cohomology for the cubic Dirac operator, Studies in Memory of Issai Schur}, Progress in Math. Vol.
\textbf{210} (2003), 69--93.

\bibitem[P]{P}
R. Parthasarathy,
\emph{Dirac operator and the discrete series}, Ann. Math. \textbf{96} (1972), 1--30.


\bibitem[V]{V}
D. A. Vogan, Jr.,
\emph{Dirac operator and unitary representations},
3 talks at MIT Lie groups seminar, Fall of 1997.

\end{thebibliography}
\end{document}